\documentclass[12pt]{article}
\usepackage{amsmath,amstext}
\usepackage{amsfonts}
\usepackage{amsthm}
\usepackage{amssymb}

\usepackage{graphicx}

\usepackage{todonotes}

\usepackage{hyperref}

\newcommand{\bbC}{\mathbb{C}}
\newcommand{\bbR}{\mathbb{R}}
\newcommand{\bbZ}{\mathbb{Z}}

\newcommand{\bbY}{\mathbb{Y}}
\newcommand{\Y}{\bbY}

\newcommand{\eps}{\varepsilon}

\newcommand{\Prob}{\mathop{\mathbb{P}}\nolimits}
\newcommand{\const}{\mathop{\mathrm{const}}\nolimits}

\newcommand{\sign}{\mathop{\mathrm{sign}}}

\newcommand{\nZ}{F}

\newcommand{\wt}{\mathbf{w}}
\newcommand{\rt}{\omega}

\newcommand{\Mm}{M_-}
\newcommand{\Mp}{M_+}

\renewcommand{\Re}{\mathop{\mathrm{Re}}}

\newcommand{\hV}{\widetilde{V}}

\newcommand{\mL}{\mathcal{L}}

\def\dd#1{\frac{\partial}{\partial#1}}

\newtheorem{theorem}{Theorem}
\newtheorem*{theorem*}{Theorem}
\newtheorem{proposition}{Proposition}

\newtheorem{Lem}{Lemma}
\newtheorem{pro}{Proposition}

\newtheorem{cor}{Corollary}
\newtheorem{conj}{Conjecture}
\newtheorem{question}{Question}
\theoremstyle{definition}
\newtheorem{definition}{Definition}

\newtheorem{rem}{Remark}

\author{A. Gordenko
\thanks{The author's work was partially supported by ANR Gromeov (ANR-19-CE40-0007).} 
\thanks{Univ Rennes, CNRS, IRMAR - UMR 6625, F-35000 Rennes, France.}
}
\title{Limit shapes of large skew Young tableaux and a modification of the TASEP process}

\begin{document}
\maketitle

\abstract{We present a survey of points of view on the problem of the asymptotic shape 
of a path between two large Young diagrams, and introduce a modification of the TASEP process 
related to it. This representation allows to write explicitly the functional, counting the asymptotics of the number of Young tableau close to a given one, as well as to see the sine-process on the boundary shape of a large random Young diagram.}

\tableofcontents

\section{Introduction}
\subsection{General background and overview of the problem}

The Young diagrams (YD for short) and notions, related to them, have been studied for a long time (for instance, see \cite{Feit, Carlitz, Concini, Edelman, Berele}). This study was motivated both by the combinatorial reasons (YD of size~$n$ correspond to partitions of number~$n$) and by the representation theory (YD of size~$n$ enumerate irreducible representations of the symmetric group~$S_n$). We denote the set of all Young diagrams of size $n$ by $\bbY_n$.

The \emph{Young graph} is an oriented graph that has Young diagrams as its vertices, and whose edges go from each YD to all YD's that can be obtained by adding a cell to the initial diagram. On the language of the representation theory, $\lambda \in \bbY_n$ is joined to all $\mu \in \bbY_{n+1}$ that are contained in the induced representation of~$S_{n+1}$, or equivalently, if the representation $\rho_\lambda$ is contained in the restriction of the corresponding representation $\rho_\mu$ to~$S_n$. The latter (together with the fact that the multiplicity of such an inclusion never exceeds one) implies that the dimension $\dim\lambda$ of the irreducible representation $\rho_{\lambda}$, associated to the YD $\lambda$, is equal to the number of paths in the Young graph that join the empty (or one-cell) diagram with $\lambda$.

A path in the Young graph, starting at the empty diagram, can be encoded by writing in each cell the number of step at which it is added, thus putting it into a bijective correspondence with a \emph{standard Young tableau}. The latter, by definition, is a way of putting numbers $\{1,\dots,n\}$ in cells of the YD of size $n$ in such a way that the numbers are increasing in each row and column, and that each number is used exactly once. (Similar construction with path going from some {\it non-empty} YD to another leads to the notion of a {\it skew Young tableau}.)

The representation theory then motivates the study of the Plancherel measure: one has 
$$
\sum_{\lambda\in \bbY_n} \dim^2 \lambda = n!,
$$
and hence the measure $\mu_n$ on $\bbY_n$, defined by $\mu_n(\{\lambda\})=\frac{\dim^2 \lambda}{n!}$, is a probability one.

\begin{figure}
\begin{center}
\includegraphics[width=7cm]{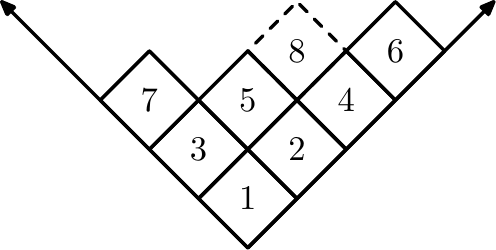}
\end{center}
\caption{Standard Young Tableau}\label{tableau}
\end{figure}

This measure gives rise to a central measure $\bar{\mu}$ on the paths on the Young graph. The {\it central measures} in general are defined in the following way. Assume that one is given a graph $G$ with the graded set of vertices $V=\bigsqcup_n V_n$, with edges joining vertices from $V_n$ to the vertices from $V_{n+1}$. By definition, a probability measure on the paths $\omega=\{\omega_n\}_{n=0}^{\infty}$, $\omega_n\in V_n$, is \emph{central} if for any $n$ and any $v\in V_n$ conditional to $\omega_{n}=v$ the initial part $\omega_0,\omega_1,\dots,\omega_{n-1},\omega_n$ of the path is distributed uniformly on all the paths that end at $v$ at the moment~$n$. 

It is easy to see that a central measure is necessarily Markovian (the future is independent from the past), and is  uniquely defined by its marginal measures $\mu_n$ defining the law of $\omega_n$. Vice versa, a sequence of measures $\mu_n$ defines a central measure, provided that they agree; the latter means that considering the law of $\omega_{n-1}$ in a uniformly chosen path leading to $v\in V_n$ and averaging with $v$ distributed w.r.t. $\mu_n$, we are getting $\mu_{n-1}$. One of the basic examples of such measures are Bernoulli ones: a random path $(x_n,y_n)$, where $x_n$ and $y_n$ are respectively the number of heads and tails after tossing of a Bernoulli coin $n$ times. Indeed, given the number $k=x_n$ of successes after tossing a coin $n$ times, all the ${n\choose k}$ possible placements of these successes are equiprobable~--- whichever was the probability~$p$ of a success.

As we have mentioned, it is known (though not evident) that Plancherel measures $\mu_n$ on sets $\bbY_n$ agree with each other and hence give rise to a central measure on the set of paths in the Young graph. This measure has forward transition probability from $\lambda\in \bbY_{n-1}$ to $\lambda'\in \Y_{n}$
$$
p_{\lambda \nearrow \lambda'} = \frac{\dim \lambda'}{n \dim \lambda}.
$$
It is easy to check that these probabilities define a Markov chain with marginal laws $\mu_n$ at time $n$, giving the backward transition probability
\begin{equation}\label{eq:backward}
\Prob(\omega_{n-1}=\lambda \mid \omega_{n}=\lambda') = \frac{\dim \lambda}{\dim \lambda'}
\end{equation}
(where $\omega_0=\emptyset, \omega_1,\omega_2,\dots$ is a path randomly chosen w.r.t. this measure) and 
hence satisfying a definition of a central measure (the relation~\eqref{eq:backward} easily implies that the distribution on the starting segments of paths coming to $\lambda\in \bbY_n$ is uniform).

A general paradigm of asymptotic combinatorics is that a large random combinatorial object often satisfies some kind of the ``law of large numbers'': if properly rescaled, it looks like a deterministic one. There are many examples of such results (for example see \cite{Propp, Vershik, Backhausz}). Of the ones related to YD, the first that we would like to mention here is the limit shape theorem, independently discovered in late 1970's by Versik and Kerov in the USSR and Logan and Shepp in the United States. 
Namely: take a random diagram $\lambda\in\Y_n$ (in French notation), contract it $\frac{1}{\sqrt{n}}$ times, and rotate it $45^{\circ}$ counterclockwise. This gives a random figure $F_n$ of unit area, placed between the rays $y=|x|$. Consider its outer boundary, extended by $y=|x|$ outside the diagram, as a graph of some 1-Lipshitz function~$f_{\lambda}$.

\begin{theorem}[Vershik, Kerov \cite{Vershik-Kerov}, Logan, Shepp \cite{Logan-Shepp}]
$f_{\lambda}$ converges in probability in $C^0$-topology to the limit function $\Omega(x)$, defined by 
$$
\Omega(x)=\begin{cases} \frac{2}{\pi}( \sqrt{2-x^2}+ x\arcsin \frac{x}{\sqrt{2}}), & |x|\le \sqrt{2},\\
|x|, & |x| \ge \sqrt{2}.
\end{cases}
$$
\end{theorem}

Now, a path in the Young graph $\omega_0\nearrow \omega_1\nearrow \dots \nearrow \omega_n$ can be also transformed in this way: rescaling it $1/\sqrt{n}$ times, we get an increasing family of figures of area $\alpha=0,\frac{1}{n},\dots,1$; again, rotating these figures by $45^{\circ}$, we can consider their (extended) outer boundaries as graphs of 1-Lipschitz functions $F_{\alpha}(x)$. This, together with the definition of the central measure, motivates the following two questions:

\begin{question}\label{q:o-l}
What can be said about a typical path from $\o$ to a given {\bf large} Young diagram~$\lambda$?
\end{question}

\begin{question}
What can be said on a random path from a given large Young diagram $\lambda_1$ to a given large Young diagram~$\lambda_2 \supset \lambda_1$? 
\end{question}

The former is already answered by the representation theory methods (see~\cite{Sniady}). The latter, its natural generalisation, was attacked with variational principle (\cite{MPP6, Sun}). We use similar approach in this paper too, though, from a different point of view.

Before proceeding, we would like to mention a few cases in which the Question~\ref{q:o-l} can be 
attacked by simple combinatorial methods. As we have already mentioned, the measure $\bar{\mu}$ is central. This implies that if we first choose a diagram $\lambda\in \Y_n$ w.r.t. the Plancherel measure and then pick a path $\emptyset=\omega_0,\omega_1,\dots,\omega_N=\lambda$ in the Young graph uniformly at random, then at each step $j$ (with $\alpha_j$ equal to the area of corresponding $\omega_j$) the diagram $\omega_j$ will be distributed w.r.t. the corresponding measure $\mu_j$. An application of the Vershik-Kerov-Logan-Shepp theorem then gives that the corresponding path $F_{\alpha}(x)$ converges in probability to the one given by rescaling of the shape $\Omega$,
$$
h_{\alpha}(x) = \sqrt{\alpha}\, \Omega(\frac{x}{\sqrt{\alpha}}).
$$
Thus, a random path to a Plancherel-random (and hence almost $\Omega$-shaped) Young diagram is given by rescaling of $\Omega$. 

\begin{figure}
\begin{center}
\includegraphics[width=9cm]{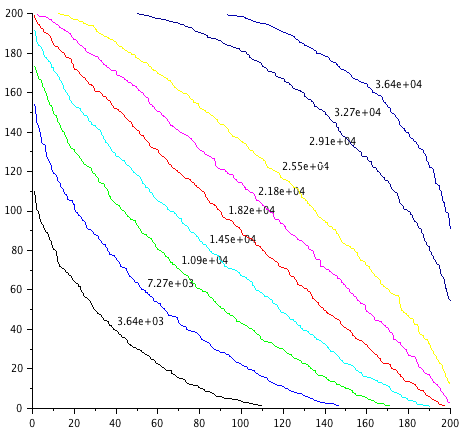}
\end{center}
\label{f:Sq}
\caption{Random YT, corresponding to the $100 \times 100$ square YD.}
\end{figure}

Next, a path, going towards a square- or rectangular-shaped Young diagram $\lambda$, can be described via the same methods as Vershik-Kerov-Logan-Shepp theorem, and it was done by Pittel and Romik in~\cite{PR}. 
Namely: the number of paths that pass through some diagram $\lambda'$ of size $j$ is a product of number of paths from $\emptyset$ to $\lambda'$ and of  number of paths from $\lambda'$ to~$\lambda$. The former can be calculated via the hook formula, and then its logarithm transformed (approximatively) into an entropy-type functional evaluated on~$\lambda'$. And the argument of Pittel and Romik says that the latter can also be calculated in this way, as the skew Young diagram $\lambda/\lambda'$ (that is, the set-theoretical difference $\lambda\setminus\lambda'$), rotated $180^{\circ}$, becomes again simply a Young diagram. Thus, one can estimate the number of paths that go through $\lambda'$, and maximizing the corresponding entropy functional, one finds the desired limit shape of the path; see Fig.~\ref{f:Sq}.

The above arguments also lead to the question of study of the number of paths from one Young diagram to the other, or, which is the same, the number of \emph{standard skew Young tableaux} of a given shape $\lambda/\lambda'$ (that is, ways of enumerating cells of $\lambda\setminus\lambda'$ in order as they appear in the path: enumeration that is increasing in each row and in each column).

It was studied in recent works by Morales, Pak, Panova and Tassy~\cite{MPP1, MPP2, MPP3, MPP5, MPP6}, using Naruse's modified hook-length formula~(\cite{NO}) and the notion of exited YD. They have conjectured (see~\cite[Conjecture 1]{MPP5}) and proved (\cite{MPP6}) that if the large diagrams $\lambda_N$ and $\lambda'_N$ have asymptotic shapes $L_{\lambda}$ and $L_{\lambda'}$ respectively (that is, the rescaled diagrams converge), then the number of paths $f^{\lambda_N/\lambda'_N}$ from $\lambda'_N$ to $\lambda_N$ has the asymptotics of the form
$$
\log F^{\lambda_N/\lambda'_N} = \frac{1}{2} n_N \log n_N + n_N \cdot c(L_{\lambda'},L_{\lambda}) + o(n_N),
$$
where $n_N=|\lambda_N/\lambda'_N|$ and $c$ is some functional. (Also, for $\lambda'$ much smaller than $\lambda$ this question was  studied in~\cite{D-F}, again, by the methods of the representation theory.)

Sun, in his work \cite{Sun}, using methods, introduced by Boutillier \cite{Boutillier}, and applying them to the \emph{beads model} (see Section \ref{s:view}), re-proved the existence of such a functional in terms of height function and also established the existence of a unique function that maximizes it.

In this paper, we present arguments that allow to write the explicit form of this functional. To state this question formally, let us give the following 
\begin{definition}
To a given large skew YT of the shape $\lambda/\lambda'$ and consisting of some number $n$ of cells, put in correspondence the function $g(t,x):[0,1]\times \bbR\to \bbR_+$, defined in the following way. For $j=0,1,\dots, n$, let $g(\frac{j}{n},x)$ be the function such that its graph is the outer boundary of the first $j$ cells of the YT, rotated by $45$ degrees and contracted by the factor $\sqrt{n}$, and let us extend this function on each of the intervals $t\in [\frac{j}{n},\frac{j+1}{n}]$ in an affine way.
\end{definition}
\begin{definition}
Consider a sequence of skew YD $\lambda_N/\lambda'_N$ of sizes $n_N$, such that the $45^{\circ}$-rotated $\frac{1}{\sqrt{n_N}}$-rescaled images of these skew YD are uniformly bounded and converge to some asymptotic shape $L/L'$. Say that the function $g(t,x)$ defines an asymptotic shape of the YT corresponding to this sequence if the functions $g_N(t,x)$ corresponding to random skew YT of the shapes $\lambda_N/\lambda'_N$ converge in probability to $g(t,x)$.
\end{definition}

\begin{conj}\label{Conj1}
\begin{itemize}
\item[$\bullet \, $]
The function $g(t,x)$, defining the asymptotic shape of a skew YT of a shape $L/L'$, maximizes the functional 

\begin{equation}\label{eq:L-def}
\mL[g]=\int_0^1 \int_{\bbR} g'_t (-\log g'_t + \log \cos \frac{\pi g'_x}{2}) \, dx \, dt - \log \frac{\pi}{\sqrt{2}}.
\end{equation}
with the boundary values $g(0,x)$ and $g(1,x)$ given by the shapes~$L$ and~$L'$ respectively.
The additive constant here is surely irrelevant for the purposes of the maximization problem, but it is important for the other conclusions.
\item[$\bullet \, $] The number $F^{\lambda_N/\lambda'_N}$ of such tableaux behaves as
\begin{equation}\label{eq:log-F}
\log F^{\lambda_N/\lambda'_N} = \frac{1}{2} n_N \log n_N + n_N\mL[g] + o(n_N),
\end{equation}
where $n_N=|\lambda_N/\lambda'_N|$ is the number of cells (recall that $g$ is chosen to be scaled to the area~$1$).
\end{itemize}
\end{conj}

Moreover, take any other continuous and almost everywhere smooth function $g_0(t,x)$, satisfying the same boundary conditions, as well as the area restrictions
$$
\forall t\in [0,1] \quad \int_{\bbR} (g_0(t,x)-g_0(0,x)) \, dx  = t.
$$
Then for any $N$ one can consider the number~$\nZ_{\eps, g_0}^{\lambda_N/\lambda'_N}$ of the YT such that the corresponding function $g$ is $\eps$-close (in the $C^0$-topology) to the function~$g_0$. And actually, the functional $\mL$ should describe the asymptotics of number such paths for any $g_0$, and this is the reason why it appears in the previous conjecture:
\begin{conj}\label{Conj2}
The number $\nZ_{\eps, g_0}^{\lambda_N/\lambda'_N}$ of YT of the shape $\lambda_N/\lambda'_N$ and $\eps$-close to the form $g_0$ has the asymptotic behaviour
$$
\log \nZ_{\eps, g_0}^{\lambda_N/\lambda'_N} = \frac{1}{2} n_N \log n_N + n_N \cdot \mL[g_0]  + o(n_N)
$$
as $n_N \rightarrow \infty$ and as $\eps \rightarrow 0$, in the sense that the double limit for the error term vanishes:
$$
\lim_{\eps\to 0} \limsup_{N\to\infty} \frac{1}{n_N} \left( \log \nZ_{\eps, g_0}^{\lambda_N/\lambda'_N} - \frac{1}{2} n_N \log n_N + n_N \cdot \mL[g_0] \right)=0.
$$
\end{conj}

\begin{rem}
With a slightly stronger notion of closeness for the skew Young diagrams to their limit forms (the limit shape boundary should be within $\const$ times the size of a cell), these statements are established in Sun's preprint~\cite{Sun}: see Definition~5.4, Theorems~7.1, 7.15 and~9.1 therein. However, we believe that these assumptions can be weakened; it seems also interesting to us that these predictions can be found by a straightforward and not too technically complicated approach.
\end{rem}

\begin{rem}
Note that we can choose another scaling normalization for the function~$g$, not necessarily choosing it to be spanned area~$1$. Let us pass to the total area~2 normalization;
formally speaking, we consider $\tilde{g}(t,x)=\sqrt{2} \,g(t,\frac{x}{\sqrt{2}})$. This normalization comes out of maya diagram consideration, see Remark~\ref{area2}. In this normalization, the functional $\mL$ can be rewritten as
\begin{equation}\label{eq:t-L}
\mL[g]=\widetilde{\mL}[\tilde{g}] := \frac{1}{2} \int_0^1 \int_{\bbR} (-\log \frac{\pi \tilde{g}'_t }{2} + \log \cos \frac{\pi \tilde{g}'_x}{2}) \, \tilde{g}'_t \, dx \, dt.
\end{equation}
The factor~$\frac{1}{2}$ here is due to the area change, while the constant $\log\frac{\pi}{\sqrt{2}}$ disappears due to the replacement of $\log g'_t$ by $\log \frac{\pi \tilde{g}'_t }{2} = \log \frac{\pi g'_t }{\sqrt{2}}$. 

It is interesting to note that in~\eqref{eq:t-L}  the derivatives in both directions of $\tilde{g}$ are multiplied by~$\frac{\pi}{2}$, possibly suggesting that $\frac{\pi \tilde{g}}{2}$ might be in some sense a more ``natural'' object.
\end{rem}

\begin{rem}\label{r:no-n-log}
A further rescaling by a factor of $n$, that is, consideration of $\tilde G(t,x):=\sqrt{n} \,\tilde g(\frac{t}{n},\frac{x}{\sqrt{n}})$, gives a figure of area $2n$, spanned during the time $n$. In these terms, the right hand side of~\eqref{eq:log-F} (except for the error term) can be written as 
\begin{equation}\label{eq:h-G}
\hat{\mL}[\tilde G] = \frac{1}{2} \int_{\bbR} \int_0^{n} (-\log \frac{\pi \tilde{G}'_t }{2} + \log \cos \frac{\pi \tilde{G}'_x}{2}) \, \tilde{G}'_t  \, dx \, dt.
\end{equation}
\end{rem}

\subsection{Modification of TASEP and computation of its entropy}\label{sec:MR}

We note that the standard skew Young diagrams (or, what is the same, paths on the Young graph) can be seen as a special kind of domino tiling on the (special part of a) hexagonal lattice. Moreover, adding a limit to this construction, one can remove the conditioning on the tiles (those not satisfying the condition have asymptotic measure zero). This is done in Section~\ref{s:view}.

This point of view, though simple, leads to interesting conclusions. It gives a strong evidence for the law of large numbers for the path between two large diagrams: there should be an asymptotic shape of a path, because there is one for the domino tilings. It allows to predict the entropy functional maximized by this path, and for that motivates an introduction of the following modified version of TASEP. 

Consider a circle with holes on it and stones placed in some of them. Every step one of the stones moves into the next hole to its right. In the classical TASEP model, all the stones which can move, do so with equal probabilities. In our case, however, the corresponding probabilities are different and depend on how freely a stone can move. Namely, we choose the probabilities of jumps in order for the entropy of the process to be maximal. We explain this in Section~\ref{s:TASEP}, and prove the following result

\begin{theorem}\label{th:TASEP}
For a circle of length $L$ with $N$ stones on it, the entropy of the corresponding topological Markov chain is equal to $\log \frac{\sin \frac{\pi N}{L}}{\sin \frac{\pi}{L}}$. The probabilities of states for the measure of maximal entropy are given by a determinantal measure whose correlation kernel is given by the projection on (any) $N$ consecutive Fourier harmonics~(out of $L$).
\end{theorem}

This process turns out to be interesting in its own: its stationary measure is determinantal, and passing to the limit it gives a handwaving explanation for the sine-process appearing on the boundary of the random large Young diagram (see Remark~\ref{r:sine-limit}) and finding the precise formula for the functional, appearing in Morales-Pak-Panova-Tassy theorem (see Conjecture~\ref{Conj1}).
In fact, we note that the formula for that functional can also be guessed by a very simple differential equations argument (see Section~\ref{s:cut}), naturally, leading to the same answer, yielding Conjectures~\ref{Conj1} and \ref{Conj2}.

\subsection{Relation to the dimer and beads models}

We encode the evolution of the modified TASEP process into a certain dimer model on the corresponding planar (mostly hexagonal) graph. Introducing a ``tax'' on edges of one of the directions ``freezes'' the model; joining it with the time rescaling, we find a nontrivial ``diagonal'' limit process. 

On one hand, such process can be explicitly described in terms of the original m-TASEP Markov chain: 
\begin{theorem*}[Theorem~\ref{t:cylinder}]
This limit process is given by coupling a maximal entropy measure for the two-sided topological Markov chain and of a Poisson process on $\bbR$ of constant intensity, providing the jump moments. The intensity of the Poisson process is equal to $e^h$, where $h$ is the entropy of the Markov chain (given by Theorem~\ref{th:TASEP}).
\end{theorem*}

On the other, using Kasteleyn theory~\cite{Kast1, Kast2}, we see that it can be described by a determinantal-type formula, and get an explicit description for its correlation kernel:

\begin{theorem*}[Theorem~\ref{stones}]
For the limit process in Theorem~\ref{t:cylinder}, the probability that the \emph{stones} are present at positions $k_1,\dots,k_n$ at times $t_1,\dots,t_m$ is equal to the determinant 
\[
\det (\tilde{K}(t_a-t_b, k_a - k_b)_{a,b=1,\dots,n}),
\]
where the kernel~$\tilde{K}$ is given by~\eqref{eq:lim-K},~\eqref{eq:lim-K-even}.
\end{theorem*}

This proposes an alternate way of establishing Theorem~\ref{th:TASEP} (see Corollary~\ref{cor:measure}). We also get a similar description for the jumping process:

\begin{theorem*}[Theorem~\ref{beads}]
For the limit process in Theorem~\ref{t:cylinder}, the common \emph{density} of the probability for the jumps at $(k_1,t_1),\dots,(k_n,t_n)$ is equal to the determinant 
\begin{equation}
\det (\tilde{K}(t_a-t_b, k_a - k_b - 1)_{a,b=1,\dots,n})
\end{equation}
for odd $N$ and to the determinant 
\begin{equation}
\det (\rt \tilde{K}(t_a-t_b, k_a - k_b - 1)_{a,b=1,\dots,n})
\end{equation}
for even $N$.
\end{theorem*}

This provides with an alternate viewpoint on the beads model considered by Boutillier~\cite{Boutillier} and Sun~\cite{Sun}, and especially on its correlation kernel. 

Finally, the relations between the jumping of stones and the dimer model also allows to provide an immediate (non-computational) explanation, why the Poissonization of the Plancherel measure is a determinantal one. Namely, this Poissonization can also be seen via the (passage to the limit in the) domino tilings on the hexagonal lattice, and the latter are known to be determinantal. This is done in Section~\ref{s:Young}.

\section{Points of view: maya diagrams, dominos, beads}\label{s:view} 

In this section we present different models, equivalent to a path in Young graph. 

We start with recalling the classical {\bf maya diagram.} Consider the real line with the holes at the points of $\bbZ + \frac{1}{2}$. 
In these holes (pictured here as white circles) black stones can be placed, each hole containing no more than one stone. 

Then one can encode the outer boundary of a YD (drawn in the Russian notation) in the following way: if the edge goes down (reading it from left to right), one places a black stone in the corresponding hole, leaving the hole empty otherwise. See Fig.~\ref{f:Maya-def}.

\begin{figure}[!h!]\label{YT}
\begin{center}
\includegraphics[scale=1]{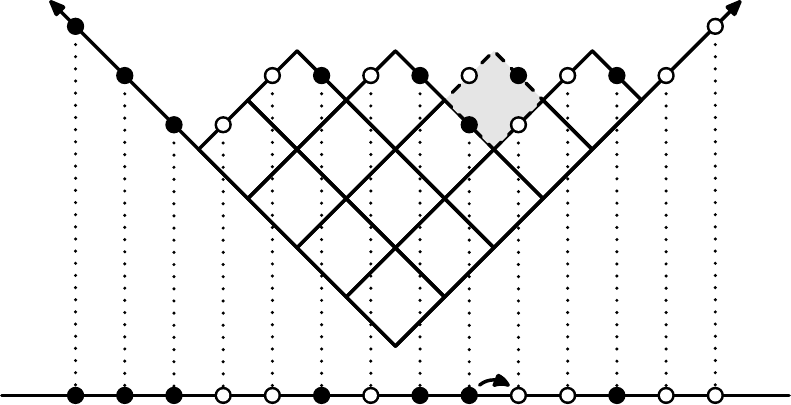}
\end{center}
\caption{Transforming a Young diagram into a maya one; an addition of a new cell (filled square) corresponds to a jump of one of the stones (shown by an arrow).}\label{f:Maya-def}
\end{figure}

\begin{figure}\label{p:models}
\begin{center}
\mbox{} \hfill 
\includegraphics[scale=0.45]{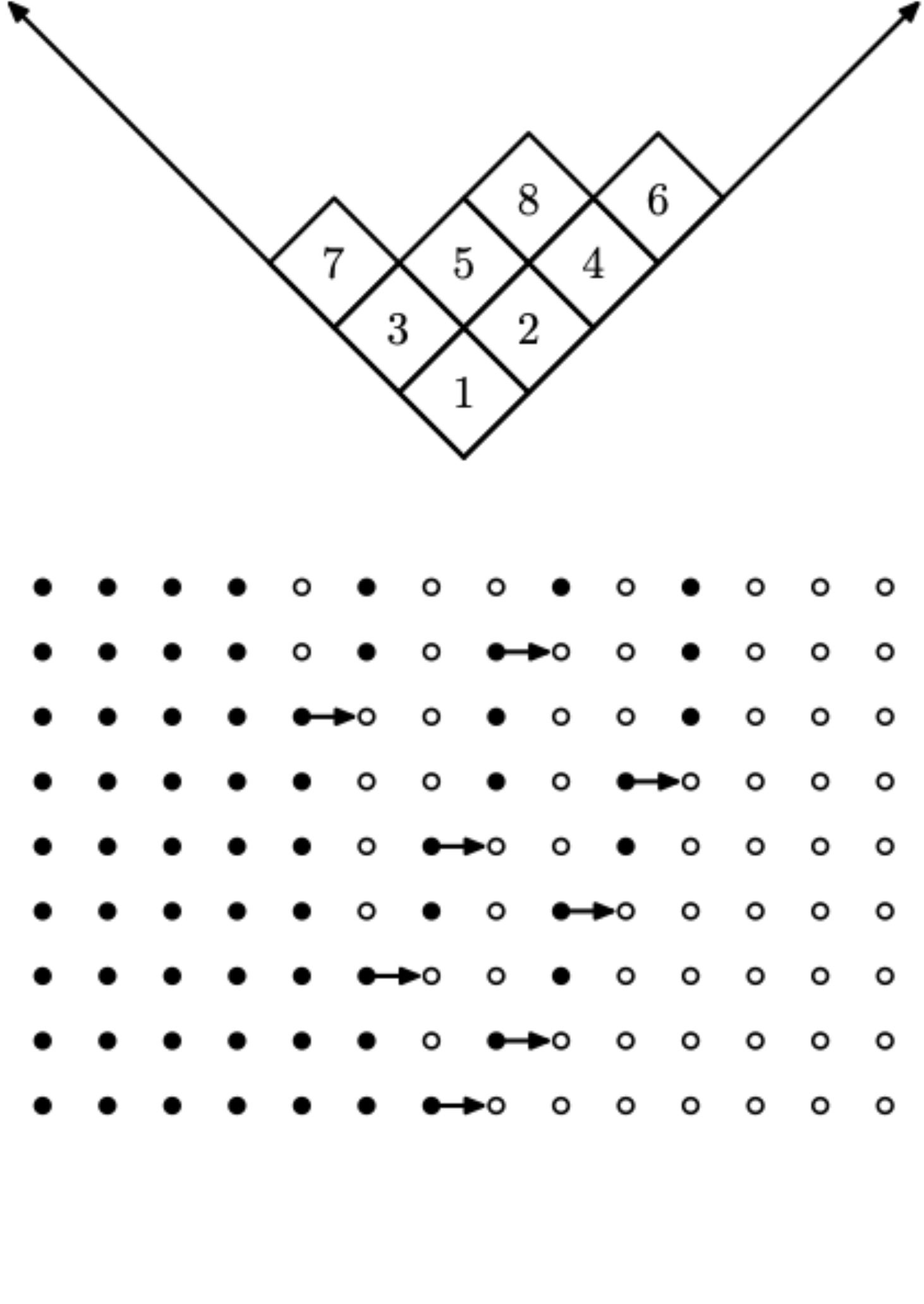} \hfill \includegraphics[scale=0.8]{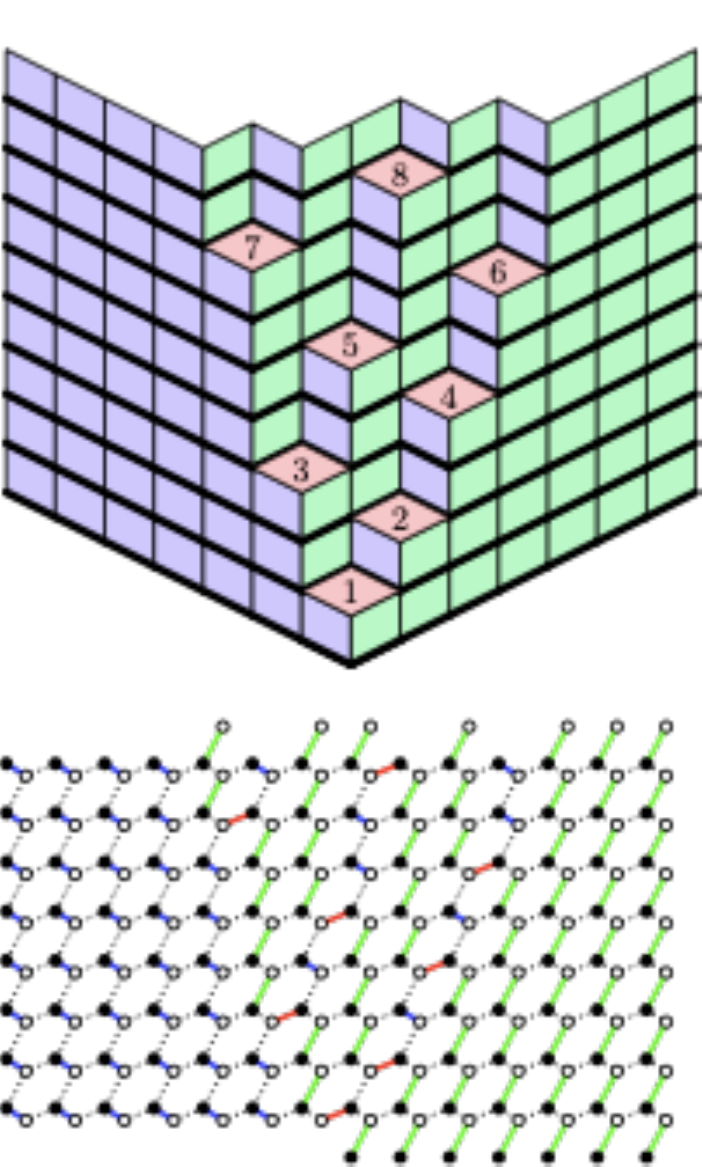} \hfill \includegraphics[scale=0.3]{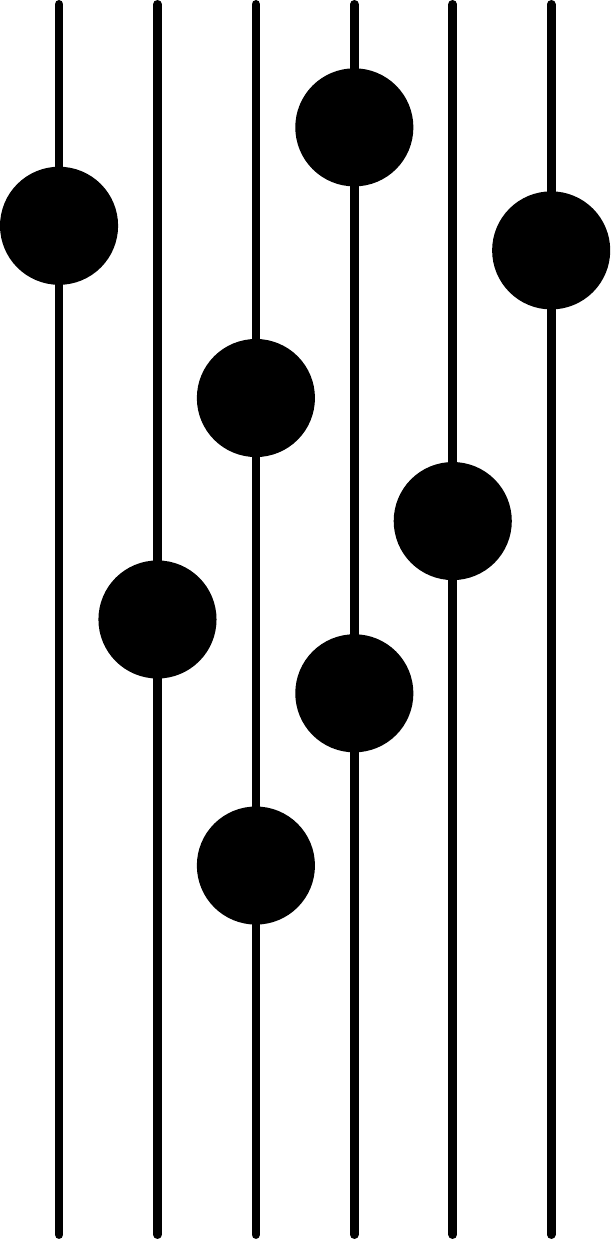}
\hfill \mbox{}
\end{center}
\caption{A Young Tableau (top left) and its encodings: maya diagram evolution (bottom left), stacked YD and lozenge tiling (top right), dimer configuration (bottom middle), beads model (bottom right).}\label{f:encodings-maya}
\label{f:encodings-tilings}
\end{figure}

One can easily see that the ``addition of a cell (provided that it can be added)'' operation in terms of YD corresponds to ``moving the stone to the next hole on its right (provided that it is empty)'' in maya diagrams' evolution. Indeed, under the addition (or removal) of a cell, the adjacent 'up' and 'down' edges on the YD border are interchanging, thus moving the corresponding stone into the empty hole next to it on the right. (See Fig.~\ref{f:Maya-def} and Fig.~\ref{f:encodings-maya}, bottom left.)

\begin{rem}\label{area2}
Note that the YD here naturally has cells of area $2$, instead of $1$ (so that their edges project to length~$1$ intervals on the $x$ axis), 
and the YD itself has area equal to $2n$. This explains why the formulae~\eqref{eq:t-L} and especially~\eqref{eq:h-G} become nicer in the corresponding normalizations.
\end{rem}

Another classical object is stacked Young diagrams. Given a path in the Young graph, one can stack the complements to the corresponding YDs, putting each of them on the top of the previous one, and considering them to be made of unit cubes instead of unit squares. This provides a 3D object, whose 3D projection gives a {\bf lozenge tiling} by lozenges corresponding to the three possible faces of the cubes; see Fig.~\ref{f:encodings-tilings}, top right. Lozenge tilings have also appeared in the works of Morales, Pak, Panova and Tassy~\cite{MPP1, MPP2, MPP3, MPP5, MPP6}, as they used an approach based on the {\bf excited diagrams}, but as this is not the one we are going to use, we will not go into further details.

Still classically, a lozenge tiling can also be seen as a {\bf dimer configuration} on the corresponding bipartite graph (a subset of the hexagonal lattice), and thus such tilings can be counted with help of the Kasteleyn theorem via the corresponding determinant. However, this approach for counting YTs has two disadvantages: first, not all the lozenge tilings correspond to the paths (one can add none or many cells on the same level), and its upper and lower boundaries depend on the shape of the skew YD that is studied (that is somewhat inconvenient).

To address the second issue, we thus will return back to the evolution of maya diagrams. We note that each such evolution can be 
encoded (in a different way!) by dimer configuration on a graph on a hexagonal lattice. Namely, the evolution of a maya diagram happens on a square lattice with the space and time coordinates $x$ and $t$ respectively. Consider all these points as black vertices, and inside each square let us add a white one. We will connect the white vertex in the square $\{n,n+1\}\times \{t,t+1\}$ to the vertices $(n,t)$, $(n,t+1)$ and $(n+1,t+1)$; see Fig.~\ref{code}.

In terms of the encoding, using the first of these edges means that there is no stone at $(n,t)$, the second one is that a stone is present and stays at this moment where it was, and the last one that the stone that was present at $(n,t)$ has jumped at this moment to the next hole. This process is illustrated on the bottom middle of Fig.~\ref{f:encodings-tilings} (red color corresponds to the edges where a stone jumps, and thus a cell is added, green edges encode empty holes, blue ones the non-jumping stones).

Note that this encoding is actually different from the one that corresponds to the stacked YDs. Indeed, though some dimer configuration via a ``backward translation'' correspond to none or many stones jumping at ones, we see that a stone here cannot jump farther than to the next hole (a possibility that appear in stacked YDs lozenge encoding), and the upper and lower bound are almost horizontal, with only hanging (green) edges describing the boundary conditions (namely, the placement of empty holes at the initial and final maya diagrams).

\begin{figure}[!h!]\label{encode}
\begin{center}
\includegraphics[scale=1]{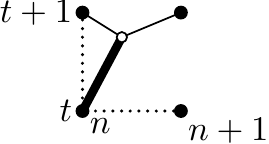} \quad 
\includegraphics[scale=1]{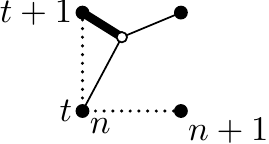} \quad 
\includegraphics[scale=1]{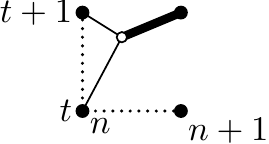}
\end{center}
\caption{Encoding: at the moment $t$ in the hole $n$ there is: no stone (left), stone that stays in the hole (center), stone that jumps to the next hole (right).}\label{code}
\end{figure}

A way of addressing the first aforementioned issue, the possibility of having two or zero jumps on the same level, is by increasing the number of levels. Namely, instead of the number of levels equal to the number of jumps~$n$, take it equal to $M\gg n$. Then, to any YT corresponds to exactly ${M \choose n}$ different configurations with at most one jump per level. On the other hand, the number of the configurations where at least two jumps happen on the same level is easily upper bounded by $\const \cdot {M \choose n-1}$, where the constant does not depend on~$M$. Hence, such configurations' fraction among all the configurations tends to 0 as $M\to\infty$. 

Contracting the picture $\const\cdot M$ times vertically and passing to the limit as $M\to\infty$, we get a continuous-time model. On one hand, the above arguments easily describe it in the initial terms: it can be obtained from independent pair of a uniform choice of a uniformly distributed YT 
(describing the places of the jumps) and a $n$-point independent choice on the time interval (describing the [rescaled] moments when these jumps occur).

On the other hand, what we thus get is a [local version of] so-called the {\bf beads model} (see Fig.~\ref{f:encodings-maya}, bottom right). It was studied in, for instance~\cite{Sun,Boutillier}; its object is a discrete subset of $\bbR\times \bbZ$, with the property that between (in the $\bbR$-direction) any its two consecutive points (``beads'') on the line $\bbR\times \{n\}$  there are points on both lines $\bbR\times \{n-1\}$ and $\bbR\times \{n+1\}$. This is exactly what we get for the placements of 
the jump sites: between any two jumps at the same place there should be the jumps in both neighboring sites; plus, for the local part of the model, the beads should satisfy the ``boundary conditions''. We will postpone the discussion on this dimer model till its use in Sec.~\ref{s:TASEP}.

\section{In search of the answer}\label{s:cut}

\subsection{Cutting the diagram}

This section is devoted to a deduction of a general form of the functional that appears in Conjectures~\ref{Conj1} and~\ref{Conj2}. We would like to emphasise that the this reasoning \emph{does not} rigorously prove the existence of such a functional. However, its existence is guaranteed by the work of Sun~\cite{Sun}, and we find it interesting that from mere fact of its existence one can deduce its explicit form by pretty straightforward (and not too technically complicated) arguments.

\subsubsection{``Horizontal'' cut}

The first step is a ``horisontal cut'' of the diagram. Namely, let YDs $\lambda' \subset \lambda$ be given. Choose a number $k$ and a sequence of ``intermediable sizes'' $n_0<n_1<\ldots<n_k$, where $|\lambda'| = n_0$, $|\lambda| = n_k$. Then the total number of paths in the Young graph from $\lambda'$ to $\lambda$ can be counted by splitting their set depending on the YDs passed at these sizes:

\begin{equation}
F^{\lambda/\lambda'} = \sum_{(\lambda_0, \ldots, \lambda_k)\in \mathcal{A}^{\lambda/\lambda'}_{n_1, \dots, n_{k - 1}}} F^{\lambda_k/\lambda_{k - 1}}\cdot \ldots \cdot F^{\lambda_1/\lambda_0}, \label{sum}
\end{equation}
where
$$
\mathcal{A}^{\lambda/\lambda'}_{n_1, \ldots, n_{k - 1}} = \{(\lambda_0, \ldots, \lambda_k) | \lambda' = \lambda_0 \subset \lambda_1 \subset \lambda_2 \subset \ldots \subset \lambda_k = \lambda, \forall i = 1, \dots, k - 1 : |\lambda_i| = n_i\}.
$$

Now, the sum \eqref{sum} is comparable with its maximum summand as it differs from the latter by the factor at most the number of summands. This number, in its turn, can be estimated as 

$$
|\mathcal{A}^{\lambda/\lambda'}_{n_1, \ldots, n_{k - 1}}| \le \prod | \bbY_{n_i} | \le \exp \Bigl( \sum_{i = 1}^{k - 1} \pi \sqrt\frac{2n_i}{3}\Bigr),
$$
where the latter inequality is due to Hardy-Ramanujan formula, $|\bbY_n|\sim \frac{1}{4n\sqrt{3}} \exp \bigl( \pi \sqrt\frac{2n}{3}\bigr)$.

Thus, we get

$$
\log F^{\lambda/\lambda'} - \log \max_{(\lambda_0, \dots, \lambda_k)\in \mathcal{A}^{\lambda/\lambda'}_{n_1, \dots, n_{k - 1}}} F^{\lambda_k/\lambda_{k - 1}}\cdot \ldots \cdot F^{\lambda_1/\lambda_0} \in [0; k\sqrt n_k].
$$

Choose $k$ much smaller than $\sqrt n_k = \sqrt{|\lambda|}$ and the sizes $n_1, \ldots, n_{k-1}$ to be ``equally spaced'' on $[n_0, n_k]$ (that is, let $n_i = n_0 + [\frac{i}{k}\cdot (n_k - n_0)])$.

It is natural to expect that for a generic skew YD of the form $\lambda/\lambda'$, its level curves at these moments slice the (rotated $\pi/4$) diagram into long and thin slices. After rescaling they should be close to the corresponding graphs $y = g(t_i, x)$, where $t_i = \frac{n_i - n_0}{n_k - n_0} \approx \frac{i}{k}$. The (total) contribution of the paths that are ``non-optimal'' will be neglectable.

\begin{figure}[!h!]
\begin{center}
\includegraphics[scale=0.4]{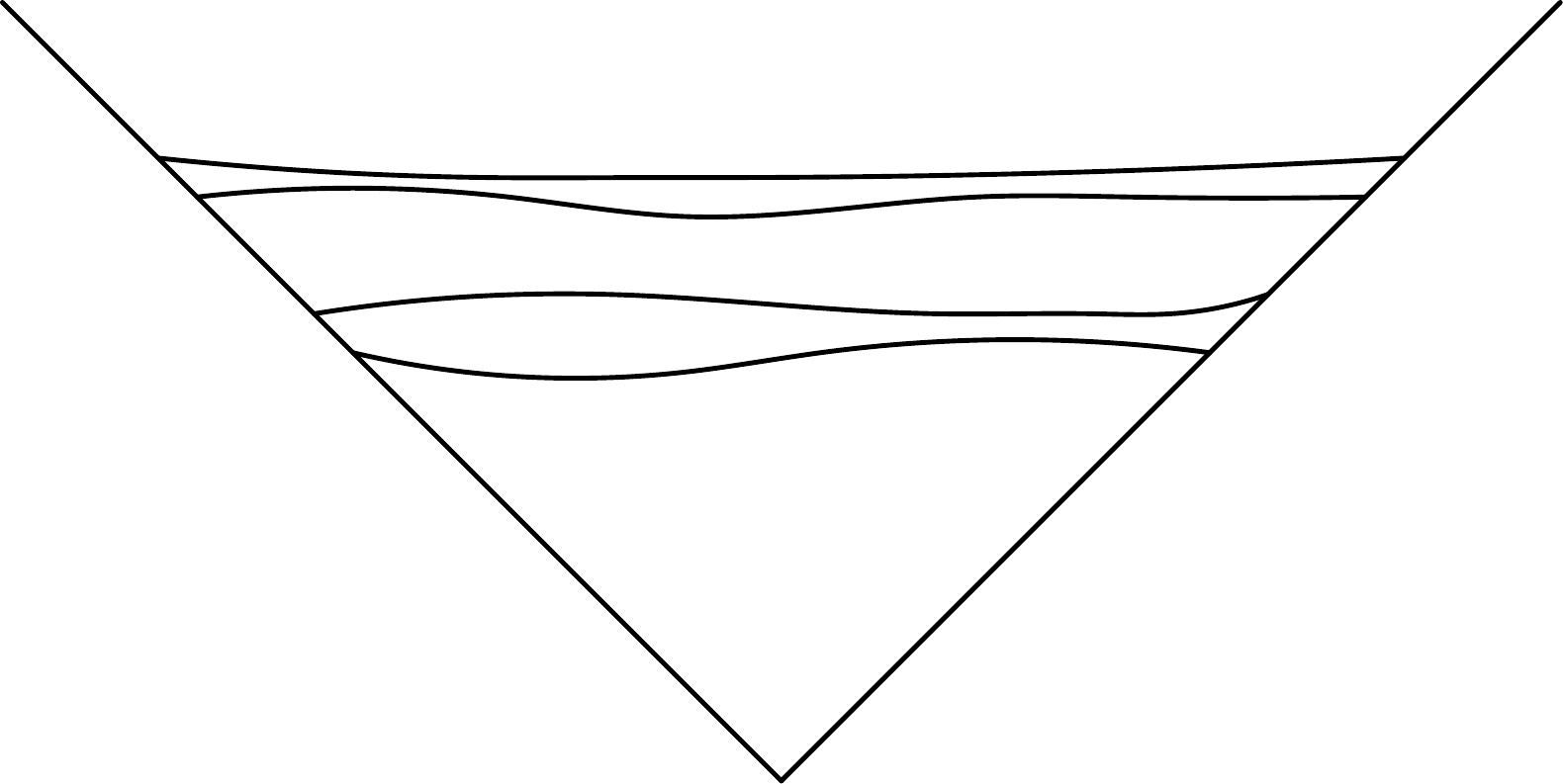} \qquad \includegraphics[scale=0.8]{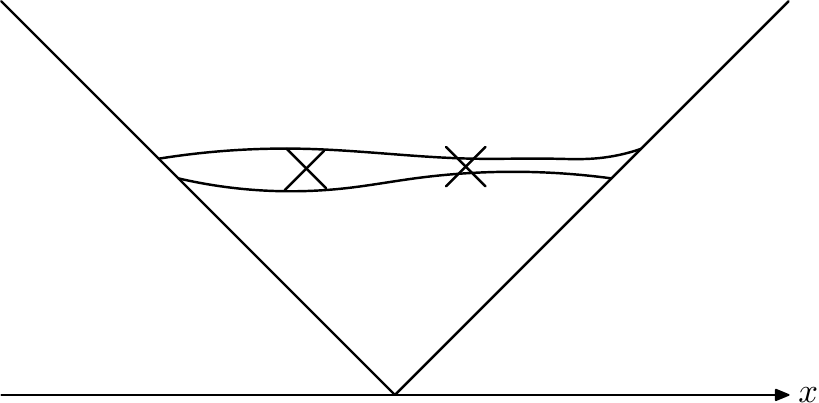} 
\end{center}
\caption{Left: ``horizontal'' cut, shapes $\lambda_i$ at the corresponding intermediate moments $t_i$. \newline
Right: ``vertical'' cut of a horizontal ``slice''~$\lambda_i/\lambda_{i-1}$.}
\end{figure}

We get an approximation (up to $o(n)$) for $\log  F^{\lambda/\lambda'}$ as
$$
\sum_{i = 1}^k \log F^{\bar \lambda_i/\bar \lambda_{i - 1}},
$$
where $(\bar \lambda_0, \bar \lambda_1, \ldots, \bar \lambda_k)$ is the index corresponding to the maximizing summand.

The same applies to the setting of Conjecture~\ref{Conj2}: given a function $g_0$, we get an approximation for $\log \nZ_{\eps, g}^{\lambda_N/\lambda'_N}$ as
$$
 \sum_{i = 1}^k \log F^{\bar \lambda_i/\bar \lambda_{i - 1}},
$$
where maximum is now taken over the set of $\lambda_i$ with the additional assumption of the (rescaled) outer boundary of~$\lambda_i$ belonging to the $\eps$-neighborhood of~$g(t_i)$.

\subsubsection{``Vertical'' cut}

Now, let us cut each ``thin'' diagram  $\bar \lambda_i/\bar \lambda_{i - 1} = D_i$ ``vertically'', choosing some points $p_{i,1}, \ldots, p_{i, m-1}$ inside~$D_i$. Let $R_{i,j}^-$ be the set of cells of $D_i$ to the lower left of $p_{i,j}$, $R_{i,j}^+$ to the upper right, and $D_{i,0}, \ldots, D_{i, m}$ the connected components of 
$$D_i \setminus \bigcup_{j = 1}^m (R_{i,j}^- \cup R_{i,j}^+) =: \tilde D_i.$$ 
One can also see $\tilde D_i$ as a skew YD:
$$\tilde D_i=\lambda_{i}^- /\lambda_{i - 1}^+, \,
\text{ where } \, \lambda_i^+ = \bar \lambda_i \cup \bigcup_{j} R_{i,j}^-, \quad \lambda_{i - 1}^- = \bar \lambda_i \setminus \bigcup_{j} R_{i,j}^+.$$

\begin{figure}[!h!]
\begin{center}
\includegraphics[scale=0.5]{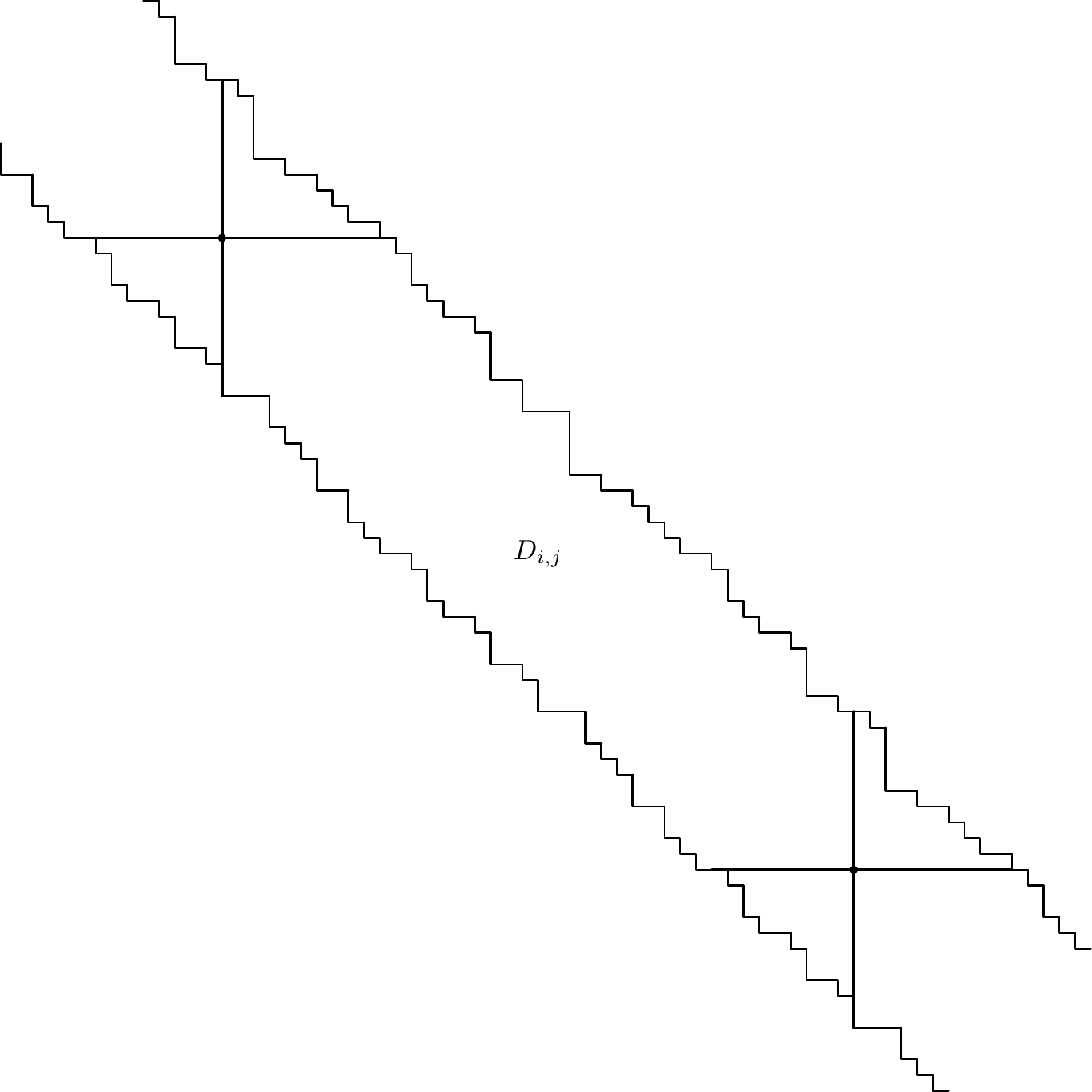} \quad 
\includegraphics[scale=0.5]{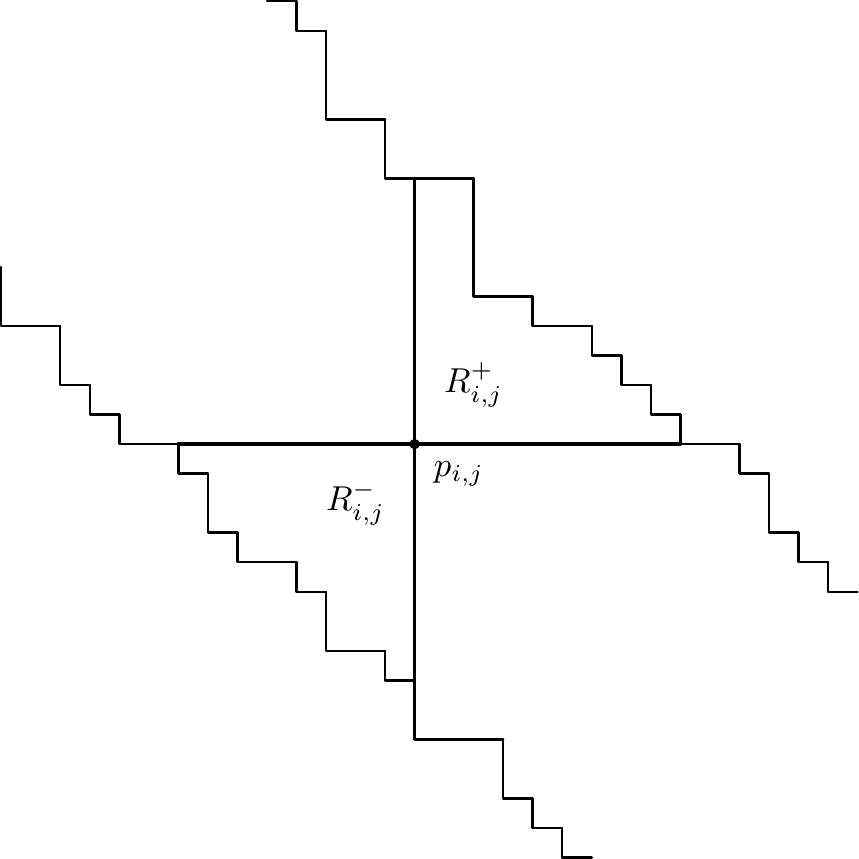} 
\end{center}
\caption{``Vertical'' cut in French notation: the domains $D_{i,j}$ (left), points $p_{i,j}$ and removed corners $R_{i,j}^{\pm}$ (right).}
\end{figure}

Consider then the map from the set of skew YT of the form $D_i$ to those of the form $\tilde D_i$: the cells are added in the same order with the parts $\bigcup_{j = 1}^m (R_{i,j}^- \cup R_{i,j}^+)$ ignored. This map is surjective: any order for $\tilde D_i$ can be completed by first adding all the cells from all $R_{i,j}^-$, then $\tilde D_i$ itself, then all $R_{i,j}^+$. On the other hand, the maximum number of preimages does not exceed $n^{\sum_j |R_{i,j}^- \cup R_{i,j}^+|}$, as we are loosing $\sum_j |R_{i,j}^- \cup R_{i,j}^+|$ numbers that do not exceed $n = n_k - n_0$. Hence, one has

$$
\log F^{D_i} - \log F^{\tilde D_i} \in [0, \sum_j |R_{i,j}^- \cup R_{i,j}^+| \cdot \log n],
$$
and thus,
\begin{equation}
\sum_{i} \log F^{D_i} - \sum_{i}\log F^{\tilde D_i} \in [0, \log n \cdot \sum_i\sum_j |R_{i,j}^- \cup R_{i,j}^+|]. \label{sum1}
\end{equation}

For a large $k$ sliced domains $D_i$ can be expected to be of width $O(\sqrt n/k)$, and thus the cutting regions $R_{i, j}^\pm$ of area $O((\sqrt n/k)^2) = O(n/k^2)$. Taking $m$ such regions per slice, we get a total effect of $O(\frac{n}{k^2}\cdot k \cdot m \cdot \log n)$ in the right side of \eqref{sum1}, and after choosing $m = o(\frac{k}{\log n})$ this error does not exceed $o(n)$.

Note now that the orderings on different components $D_{i,j}$ of $\tilde D_{i}$ are completely independent. That is, let $YT^D$ stay for the (skew) standard YT of the shape $D$. Consider the map 
$$
P_i:YT^{\tilde D_{i}} \to \prod_j YT^{D_{i,j}},
$$
defined by restricting order of appearance of cells in $\tilde D_i$ on each subdiagram $D_{i,j}$. It is easy to see that this map is exactly $R_i$-to-one, where $R_i$ is the multinomial coefficient
$$
R_i = \binom{|\tilde D_i|}{|D_{i,1}|, \ldots, |D_{i,m}|} = \frac{|\tilde D_i|!}{|D_{i,1}|! \dots  |D_{i,m}|!}.
$$
Hence, 
$$
\log F^{\tilde D_i} = \sum_j \log F^{D_{i,j}} + \log \frac{|\tilde D_i|!}{|D_{i,1}|! \dots  |D_{i,m}|!}.
$$

Meanwhile, from Stirling's formula we have 
$$
\log \binom{|\tilde D_i|}{|D_{i,1}|, \ldots, |D_{i,m}|} = \sum_{j = 1}^m |D_{i,j}| \cdot \left( - \log \frac{|D_{i,j}|}{|\tilde D_i|}\right)+o(|D_i|),
$$
(as the sizes of $|D_{i,j}|$ tend to infinity at least as $\log n$); we thus get an approximation
\begin{equation}
\log F^{\lambda/\lambda'} = \sum_i \sum_j \left[ \log F^{D_{i,j}} + |D_{i,j}| \cdot (- \log \frac{|D_{i,j}|}{|\tilde D_i|}) \right] + o(n). \label{logsum}
\end{equation}

Again, instead of all the paths we can consider only the paths that ``resemble'' a graph of a function~$g$. For such a path $\bar \lambda_1, \ldots, \bar \lambda_{k - 1}$, the skew YDs $D_{i,j}$ look like parallelograms of horizontal length $\sqrt n \cdot (x_{i,j} - x_{i,j-1})$ (where the vertical point  $p_{i,j}$ has $x$-coordinate $x_{i,j}$) 
and of width $\sqrt n \cdot g'_t(t_i, x_{i,j})$ and with the slope $\tan \alpha = g'_x (t_i, x_{i,j})$.

\begin{figure}[!h!]\label{encode1}
\begin{center}
\includegraphics[scale=0.7]{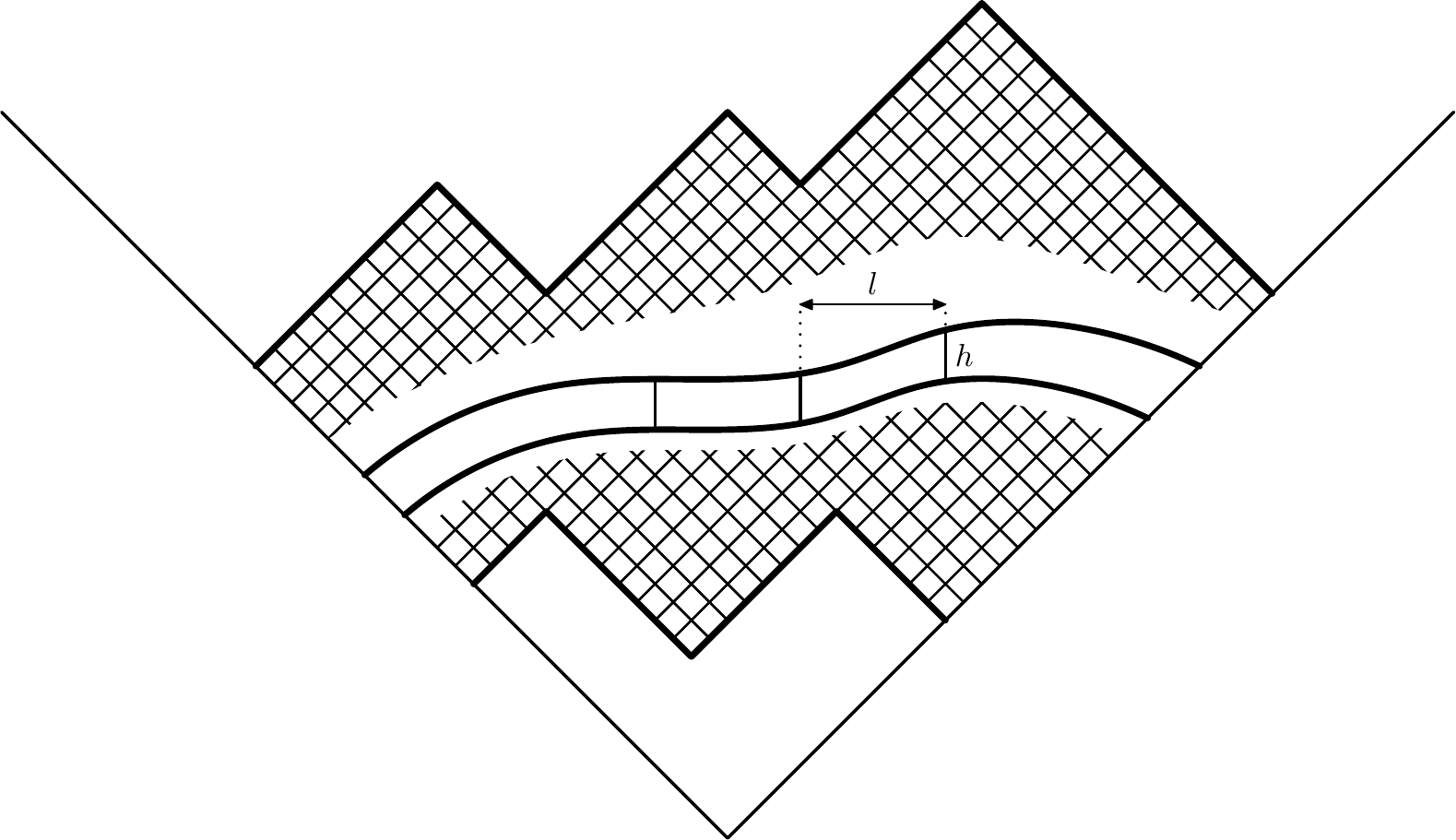} \qquad
\includegraphics[scale=0.9]{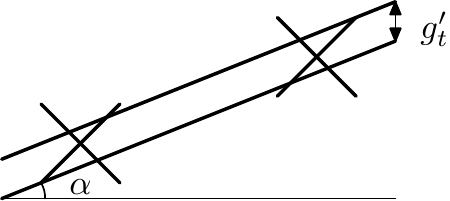}
\end{center}
\caption{``Vertical'' cuts}
\end{figure}

Hence, to transform the formula \eqref{logsum} to the desired integral form, we have to estimate the logarithmic number $Z(\alpha, h, l) = \log F^{\Pi_{\alpha, h, l}}$ of skew YT in a parallelogram of length $l$, height $h$, where $1\ll h\log h \ll l$, going under a slope~$\tan \alpha$.

\subsubsection{Parallelograms approximation}
The same arguments as before imply the following two conclusions should for large $l \gg h\log h \gg 1$:
\begin{itemize}
\item $Z(\alpha, h_1 + h_2, l) \approx Z(\alpha, h_1, l)  + Z(\alpha, h_2, l) $~--- from adding an additional ``intermediate moment'', cutting the parallelogram ``horizontally'';
\item $Z(\alpha, h, l_1 + l_2) \approx Z(\alpha, h, l_1) + Z(\alpha, h, l_2) + \log \binom{h(l_1 + l_2)}{hl_1,~hl_2}$~--- from adding an additional midpoint, ``vertically cutting in independent domains''.
\end{itemize}

The latter approximation can be further rewritten as
$$
Z(\alpha, h, l_1 + l_2) \approx Z(\alpha, h, l_1) + Z(\alpha, h, l_2)  + h (l_1 + l_2) \log (l_1 + l_2) - h l_1 \log l_1 - h l_2 \log l_2.
$$

Considering the difference $\tilde Z(\alpha, h, l) := Z(\alpha, h,l) - h l \log l$, we see that it is thus (approximately) additive in both $h$ and $l$.
Hence, it is natural to expect it to behave like
$$
\tilde Z (\alpha, h,l) = A(\tan \alpha) h l + o(hl),
$$
where $A(\tan \alpha)$ is a constant, depending only on the slope $\tan \alpha$. Thus, we get a prediction
\begin{equation}\label{eq:Z-A}
Z(\alpha, h,l) = h l \log l + A(\tan \alpha) h l + o(hl).
\end{equation}

As a concluding remark, note that due to the vertical symmetry (in the Russian notation) the function $A(\cdot)$ should be even. 

\subsubsection{Integral formula}\label{sss:integral}

Plugging~\eqref{eq:Z-A} back to \eqref{logsum}, we get an asymptotic expression for the number of $g$-shaped skew~SYT:
\begin{equation}
\log \nZ_{\eps,g}^{\lambda/\lambda'} = \sum_{i,j} |D_{i,j}| \cdot \left[ \log l_{i,j} + A(\tan \alpha_{i,j}) + \log \frac{|D_i|}{|D_{i,j}|}\right] + o(n),.  \label{logsum1}
\end{equation}
Here $o(n)$ is understood in the sense of a double limit as $\lim_{\eps\to 0} \limsup_{n\to\infty}$, 
we denote by $l_{i,j}$ is the (horizontal) length of the ``parallelogram'' $D_{i,j}$ and by $\tan \alpha_{i,j} = g'_x(t_i, x_{i,j})$ its slope. 
The height $h_{i,j}$ of $D_{i,j}$ after rescaling by $\sqrt n$ can be approximated as
$$
\frac{h_{i,j}}{\sqrt n}\approx g'_t (t_i, x_{i,j})\cdot (t_i - t_{i - 1});
$$ 
as $t_i - t_{i-1} = \frac{n_i - n_{i - 1}}{n}$, we get
$$
h_{i,j} \approx g'_t(t_i, x_{i,j}) \cdot \frac{n_i - n_{i - 1}}{\sqrt n}.
$$ 

As $|D_{i,j}|\approx l_{i,j}h_{i,j}$, $|D_i| = n_i - n_{i - 1}$, we can write the expression in the right hand side of~\eqref{logsum1} as
\begin{multline}
\log l_{i,j} + A(\alpha_{i,j}) + \log \frac{|D_i|}{|D_{i,j}|} \approx \\
\approx \log l_{i,j} + \log{(n_i - n_{i - 1})} - \log{\left( l_{i,j} \frac{n_i - n_{i - 1}}{\sqrt n} g'_t(t_i, x_{i,j})\right)} + A(g'_x(t_i, x_{i,j})) \approx \\
\approx \frac{1}{2} \log n - \log g'_t(t_i, x_{i,j}) + A(g'_x(t_i, x_{i,j})).
\end{multline}

Multiplying by $|D_{i,j}|\approx (t_i - t_{i - 1})(x_{i,j} - x_{i,j-1})\cdot g'_t(t_i, x_{i,j})\cdot n$, and adding up, we finally get the desired
\begin{multline}
\log \nZ_{\eps,g}^{\lambda/\lambda'} = \sum_{i,j} |D_{i,j}| \cdot \left[ \frac{1}{2} \log n - \log g'_t(t_i, x_{i,j}) + A(g'_x(t_i, x_{i,j}))\right] = \\
= \frac{1}{2} n \log n + n \left[ \iint (- \log{g'_t} + A(g'_x))\cdot g'_t \,dx \,dt + o(1)\right].
\end{multline}
This is exactly the statement of Theorem~\ref{Conj2}. Taking the maximum over the possible shapes~$g$ of the skew SYT and referring to the variational principle then implies Theorem~\ref{Conj1}. Indeed, if $g_0$ is the maximizing function for the functional $\mL$ (it is easy to see that it is concave, so $g_0$ is unique), any other $g$ will correspond to the exponentially smaller number of paths.

We conclude this paragraph by reminding that all the arguments therein are non-rigorous, serving as a good motivation for these conjectures, but not as a rigorous proof.

\subsection{Differential equation}\label{s:functional}

The discussion on the previous section implies that the number of $g$-shaped skew YT of area~$n$ should be asymptotically diven by the formula 
$$
\log \nZ_{\eps,g}^{\lambda/\lambda'} = \frac{1}{2} n \log n + n \cdot \mL[g] + o(n),
$$
where
\begin{equation}\label{eq:L}
\mL[g]=\int\limits_0^1\int\limits_\bbR(-g'_t \log g'_t + g'_t A(g'_x)) \, dx\, dt,
\end{equation}
and the function $A(\cdot)$ is yet to be determined. Also, the limit shape of a skew YT of a given large form should be an extremal of this functional.

\begin{rem}
This is not an immediate conclusion, as we have used that the parameter $t$ corresponds to the part of area filled, and hence the allowed functions $g$ are only those satisfying $\forall t \in [0;1]: \int (g(t, x) - g(0,x)) \, dx = t,$ or, equivalently, for sufficiently smooth functions,  
\begin{equation}\label{eq:A}
\forall t \in [0;1]: \int g'_t(t, x) , dx = 1.
\end{equation}

Thus $g$ is immediately an extremum of $\mL$ only on the space of functions, given by~\eqref{eq:A}. However, for any (increasing in $t$) function $g(t,x)$ we can consider its time reparametrization $\tau = \phi (t)$:

$$
\phi(t) = \int\limits \bbR (g(t, x) - g(0,x))\, dx,
$$
and the corresponding function $\tilde g(\tau, x) = g(\phi^{-1}(t), x)$.

It is easy to see that the $A$-part of the functional $\mL$, that is, $\iint A(g'_x) g'_t \, dx \,dt$ stays unchanged by such a reparametrization. Meanwhile, 

$$
\iint - \tilde g'_t \log \tilde g'_t \,dx \,dt = \iint -g'_t \log g'_t \,dx\,dt - \int \phi' \log\phi' \,dt, 
$$
and as $- \int \phi' \log\phi' \,dt \ge 0$, and strictly $>0$ for all non-identity $\phi$ (as $\phi(0) = 0, \phi(1) = 1$ and Jensen inequality), the maximum of $\mL$ is attained on a function $g$ with uniform growth.

\end{rem}

It turns out that these observations suffice to reconstruct $A(\cdot)$.

Namely, as we have mentioned in the introduction, a skew YT of a shape following from Vershik-Kerov-Logan-Shepp asymptotics is given by a family of its rescalings:

\begin{equation}
\Omega(t,x) = \sqrt{t} \cdot \Omega\left( \frac{x}{\sqrt t}\right). \label{VKLS} 
\end{equation}

This is an extremal of a functional $\mL$, and thus it should satisfy the Euler-Lagrange equations:

\begin{equation}
\dd t L'_{g'_t}(g'_t, g'_x) + \dd x L'_{g'_x}(g'_t, g'_x) = 0, \label{EuLa} 
\end{equation}
where 

\begin{equation}
L(p_t, p_x)  = - p_t \log {p_t} + A(p_x) \cdot p_t. \label{lagrangian} 
\end{equation}

As $\Omega(t,x)$ given by \eqref{VKLS} is an explicit function, we can plug it in \eqref{EuLa} and interpret it as a differential equation for unknown $A(\cdot)$.

\begin{pro}
Let $A: [-1, 1] \rightarrow \bbR$ be an even function, $C^2$-smooth on $(-1, 1)$. Then $\Omega(t,x)$ satisfies the Euler-Lagrange equation for the functional $\mL[\cdot]$ if and only if
$$
A(p_x) = \log \cos \frac{\pi p_x}{2} + C,
$$
where $C$ is a constant.
\end{pro}

\begin{proof}

Let us first rewrite the Euler-Lagrange equation \eqref{EuLa} using the explicit form of the Lagrangian \eqref{lagrangian}:
\[
L'_{p_t} = -\log p_t - 1 + A(p_x),
\]
\[
L'_{p_x} = p_t \cdot A'(p_x),
\]
and thus \eqref{EuLa}  becomes 

\[
\dd{t} (A(g'_x) - \log g'_t - 1) + \dd{x}(g'_t A'(g'_x)) = 0,
\] 
and hence
\begin{equation}
A'' (g'_x) g''_{xx} (g'_t)^2 + 2 A'(g'_x) g''_{xt} g'_t - g''_{tt} = 0. \label{EuLa2}
\end{equation}

Now, for $g(t,x) = \Omega(t,x)$ we have 
$$
\Omega'_x(t,x) = \frac{2}{\pi} \arcsin\frac{x}{\sqrt t}, \quad \Omega'_t (t,x) = \frac{\sqrt{t - x^2}}{\pi t}.
$$ 
Thus $\frac{x}{\sqrt t} = \sin\frac{\pi \Omega_x}{2}$; substituting this into~\eqref{EuLa2}, we get
\[
4 A''(\Omega_x) (1 - \sin^2 \frac{\pi \Omega_x}{2}) - 4 \pi A'(\Omega_x) \sin \frac{\pi \Omega_x}{2} \cos \frac{\pi \Omega_x}{2} - \pi^2 (2 \sin^2 \frac{\pi \Omega_x}{2} - 1) = 0.
\]

Finally, making a change of variable $\xi = \Omega_x$, we get a linear inhomogeneous differential equation
\[
G'(\xi) \cdot 4 \cos^2 \frac{\pi \xi}{2} - 2\pi G{\xi} \cdot \sin(\pi \xi) + \pi^2 \cos{(\pi\xi)} = 0,
\]
for the derivative $G(\xi) = A'(\xi)$ (that should be odd as $A(\cdot)$ is even). A straightforward 
computation then shows that it admits a unique odd solution
\[
G'(\xi) = - \frac{\pi}{2} \tan \frac{\pi\xi}{2},
\]
and integrating it, we get the desired form for an even solution $A(\cdot)$:
\[
A(\xi) = \log \cos \frac{\pi\xi}{2} + C.
\]
We denote the ``constant-free'' part by $A_0(\xi) := \log \cos \frac{\pi\xi}{2}$.

\end{proof}

\subsection{Determining the constant}

Note that replacing $A_0$ by $A_0+C$ in \eqref{lagrangian} changes the total value of the functional $\mL$ by

\[
\int\limits_0^1 \int\limits_\bbR C \cdot g'_t \, dx \,dt = C \cdot \int\limits_\bbR (g(1,x) - g(0, x)) \, dx = C, 
\]
as we choose the normalisation of $g$ to give the figure of total area $1$. 
This explains why the constant $C$ is irrelevant to the problem of asymptotic shape: replacing $\mL$ by $\mL + C$ doesn't change its extremals. However, the value of $C$ is important for the ``total number of paths'' asymptotics of Conjecture \ref{Conj1}, and it can be found again with help of VKLS shape $\Omega(t,x)$. 

Namely, one has $\sum_{\lambda \in \bbY_n} \dim^2 \lambda = n!$. At the same time, the number of summands grows subexponentially, $|\bbY_n|\le \exp(c\cdot \sqrt n)$. Hence for most YD $\lambda$ in the sense of the Plancherel measure, $\dim \lambda$ is close to $\sqrt{n!}$ on the logarithmic scale:

$$
\forall r ~~~ \mu_n\left( \left\{  \lambda : \dim \lambda \le \sqrt\frac{n!}{r |\bbY_n|}\right\}\right) \le \frac{1}{r},
$$
hence for YDs with probability at least $1 - \frac{1}{r} $

$$
\sqrt{n!} \ge F^{\lambda/\o} = \dim \lambda \ge \frac{\sqrt{n!}}{\sqrt{n|\bbY_n|}}.
$$

The asymptotic shape of such diagrams is given by $\Omega(x)$, and of the corresponding YT by $\Omega(t,x)$. As 
\[
\log{n!} = \frac{1}{2} n \log n - \frac{1}{2}n + o(n)
\] 
and 
\[\log F^{\lambda/\lambda'} = \frac{1}{2}n \log n + n \cdot \mL[g] + o(n),
\] 
we have
$$
\mL [\Omega(t,x)] = - \frac{1}{2}.
$$

Calculating the corresponding double integral explicitly (we omit the straightforward calculations here), one finally gets the value
$$
C =  - \log\frac{\pi}{\sqrt 2}.
$$

\section{Modified TASEP and the discrete sine-process}\label{s:TASEP}

In this section we introduce the ``local'' maya model, briefly described in \S\ref{sec:MR}, and use it to re-obtain the functional of 
Conjecture~\ref{Conj1} from a different angle of approach.

\subsection{Markov chain and the discrete sine-process}

Namely, consider an analog of maya diagram on the circle instead of a real line, formed of some number $L$ of holes. The rule ``stone jumps to its right'' is then rewritten as ``stones jump in the positive direction''; see Fig.~\ref{f:mTASEP}. As the total number of stones is preserved by a jump, this total number (that we denote~$N$) is invariant under such a dynamics. Thus, for any $L$ and $N$ we get a topological Markov chain. 

It is quite similar to the TASEP ({\bf t}otally {\bf a}symmetric {\bf p}rocess), however, for the classical TASEP model all the stones that can jump do so equiprobably. We are concerned with the \emph{topological} entropy of this chain (as we are interested in counting all the possible trajectories for the YTs). Thus, we are interested in the \emph{measure of maximal entropy} for this chain (and the corresponding Markov shift as a dynamical system), thus modifying the jumping probabilities accordingly. An immediate observation is that the stones are more likely to jump if this jump does not reduce the number of degrees of freedom, creating a tightly packed group of stones, as this is likely to reduce the number of options on the next steps. In particular, the probabilities of such 
``crumpled'' states will be reduced (contrary to the classical TASEP, where all the possible states are equiprobable).

\begin{figure}\label{cut}
\begin{center}
\includegraphics[scale=1.5]{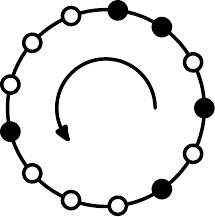}
\end{center}
\caption{TASEP: black stones are allowed to move only in the positive direction.}\label{f:mTASEP}
\end{figure}

Our main (formal) result, Theorem~\ref{th:TASEP}, describes the topological entropy and the maximal entropy measure for this topological Markov chain:
\begin{theorem*}
For any $L,N$, the entropy of the topological Markov chain defined above is equal to 
$$
h = \log \frac{\sin\frac{\pi N}{L}}{\sin\frac{\pi}{L}}.
$$
The corresponding measure of maximal entropy is a determinantal one; the correlation kernel, giving the distribution of possible states, 
is given by the projection on (any) $N$ consecutive Fourier harmonics on the length $L$ discrete circle.
\end{theorem*}
Postponing its proof till~\S\ref{ss:proof}, let us discuss the relation of this process to our main theme. Namely, we use it to describe a possible local evolution over a (large) part of it, that we consider to be winded to the circle, in the same way as parts of (hexagonal or square) lattices are winded to a torus (see, e.g.~\cite{Propp}). Thus, for a large YD and the corresponding maya diagram evolution, a local part of it can be modelled by taking a large circle and filling it with the same proportion of stones that are observed at this point of space and time. 

Now, the corresponding height function increases by~$1$ at the stone and decreases by~$1$ at each empty hole. Hence, while going around the circle it increases by $N-(L-N)=2N-L$ (so, formally speaking, this is a multi-valued function with a logarithmic monodromy). This corresponds to the slope of $\frac{2N-L}{L}$, that has a meaning of $g'_x$ (if this circle is but a small part of a large YT). Denoting $p:=\frac{N}{L}$ the density of the stones, we see that $p=\frac{g'_x+1}{2}$, thus $\frac{\pi N}{L} = \pi p = \frac{\pi}2 (1+g'_x)$ and hence that this (``local'') entropy can be rewritten as
$$
\log \frac{\sin\frac{\pi N}{L}}{\sin\frac{\pi}{L}} = \log \cos \frac{\pi g'_x}{2} - \log \sin \frac{\pi}{L}.
$$

On the other hand, for large $L$ we have $\sin \frac{\pi}{L} \approx \frac{\pi}{L}$, while $\frac{2}{L}$ is a speed at which the height function increases in average per one iteration of the process (a jump increases it in two cites, see Figure~\ref{YT}). Hence $\sin \frac{\pi}{L}\approx \frac{\pi g'_t}{2}$. 
Gluing independent local ``circled'' pieces together (in the same way as we did it in Section~\ref{sss:integral}), we see that the global number of [$g$-shaped] paths will be given by an integral of 
$$
\log \cos \frac{\pi g'_x}{2} - \log \frac{\pi g'_t}{2}.
$$
That is exactly what is suggested by Conjecture~\ref{Conj2} in the form of~\eqref{eq:h-G} in Remark~\ref{r:no-n-log}; the coefficient~$\frac{1}{2}$ comes from 
the fact that cells are of area~$2$, see Remark~\ref{area2}).

A concluding~--- and still informal~--- remark in this paragraph is that the consideration of this process leads to a handwaving explanation of the sine process appearing as the local shape of a (Plancherel)-random Young diagram (see~\cite[Theorem~3]{BOO}). Indeed, it is quite natural to expect that the local behaviour can be approximated by the corresponding maximal entropy measure. And there is the following 
\begin{rem}\label{r:sine-limit}
As we consider longer and longer circles, filled with a given limit density of stones $\frac{N_j}{L_j}\to a \,\in (0,1)$, the corresponding maximal entropy measures converge to the sine process. Indeed, their correlation kernels are projections on consecutive $N_j$ harmonics out of $L_j$, and this kernel converges to the kernel of projection of the Fourier transform to the arc that takes $a$-th part of the circle (Fourier-dual to~$\bbZ$). That kernel is exactly the one of the sine process, 
\[
K(k,l;a)=K(k-l,a)=
\begin{cases}
\frac{\sin \pi a (k-l)}{\pi(k-l)}, & k\neq l \\ 
a, & k=l.
\end{cases}
\]
\end{rem}

\subsection{Proof of Theorem \ref{th:TASEP}}\label{ss:proof}

Let $L,N$ be fixed, and consider the set of states of the topological Markov chain. Recall that the topological entropy is the logarithm 
of the spectral radius of the transition matrix~$T$, and the corresponding eigenvalue is real and positive. Moreover, if $v$ and $u$ are the corresponding non-negative left and right eigenvectors, the probabilities of states for a maximal entropy measure (``Parry measure'', see~\cite{Pollicott, Parry}) are given by the normalization of the vector with the coordinates~$u_s v_s$. 

Consider first the case of $N$ odd (this case is slightly simpler). The states of the Markov chain are enumerated by ${L \choose N}$ possible arrangements of the stones. Take a space $V=\bbR^L$; for any state of the chain, let $k_1<\dots<k_N$ be the numbers of stone-filled holes on the circle, and put in correspondence to it the element $v_{k_1,\dots, k_N}:=e_{k_1}\wedge \dots \wedge e_{k_N}\in\Lambda^N V$.

The transition matrix $T$ then acts on $\Lambda^N V$ in the following way. Let $C$ be operator that cyclically permutes the base of $V$, that is, $C(e_k):=e_{k+1 \textrm{ mod } L}$. Then 
\begin{equation}\label{eq:P-C}
T(e_{i_1}\wedge \dots \wedge e_{i_N}) = C(e_{i_1}) \wedge e_{i_2} \wedge \dots \wedge e_{i_N} + 
e_{i_1} \wedge C(e_{i_2}) \wedge \dots \wedge e_{i_N} + \dots + e_{i_1} \wedge \dots \wedge e_{i_{N-1}}  \wedge C(e_{i_N}). 
\end{equation}
Indeed, application of $C$ to $e_{i_k}$ corresponds to a possible jump of this stone; if the next hole, $i_k+1$-th, is filled, the jump is forbidden, that corresponds to the vanishing of the corresponding wedge product in the right hand side. Finally, as $N$ is odd, even if $i_N=L$ and thus the jump of this stone to the position~$1$ leads to the cyclic re-enumeration, this doesn't affect the final result, as 
$$
e_{i_1}\wedge e_{i_2}\wedge \dots\wedge e_{i_N}= e_{i_2}\wedge \dots\wedge e_{i_N}\wedge e_{i_1}.
$$

The right hand side of~\eqref{eq:P-C} is simply the operator $C\wedge E\wedge\dots\wedge E$, where $E$ is the identity operator on~$V$. Hence, its eigenvalues are sums of any $N$ different eigenvalues of~$C$, and the eigenvectors are the wedge products of the corresponding eigenvectors of~$C$. The eigenvalues of $C$ are $L$-th power roots of unity $\lambda_k=\exp(2 \pi i k / L)$, and the corresponding eigenvectors are discrete Fourier harmonics $v_k=\sum_j \exp(-2\pi i k j/L) e_j$.

Among the sums of $N=2m+1$ different $\lambda_k$'s, the maximal in absolute value are the ones corresponding to the consecutive (on the circle $\mod L$) eigenvalues; in particular, the positive and maximal one is 
$$
r=\lambda_{-m}+\dots+\lambda_m = \frac{e^{2\pi i \cdot(m+\frac{1}{2})/L} - e^{-2\pi i \cdot(m+\frac{1}{2})/L}}{e^{\pi i/L} - e^{-\pi i/L}}= \frac{\sin \frac{\pi N}{L}}{\sin \frac{\pi}{L}}.
$$
The topological entropy $h$ is equal to its logarithm, thus proving the entropy part of the theorem. 

Now, consider the corresponding eigenvector.
It is given by the product $v_{-m}\wedge\dots\wedge v_m\in\Lambda^N V$. Moreover, the right eigenvector $u$ has the same coordinates (replacing of $C$ by $C^*=C^{-1}$ leads to the same answer), though we prefer to conjugate its elements: 
$$
u=\overline{v_{-m}}\wedge\dots \overline{v_{m}}.
$$
Then, the probabilities of every state $k_1<\dots<k_N$ are proportional to
\begin{equation}\label{eq:K}
\det ((v_{i})_{k_j})_{i=-m,\dots,m \atop j=1,\dots,N} \cdot \det ((\overline{v_{i'}})_{k_j})_{i'=-m,\dots,m \atop j=1,\dots,N} = \det ((K)_{k_{j},k_{j'}})_{j,j'=1,\dots, N},
\end{equation}
where $K=\sum_{i=-m}^{m} v_i \cdot (v_i)^*$ is the projection operator on the subspace $\langle v_{-m},\dots,v_m\rangle\subset V$.
As $K$ is the rank $N$ orthogonal projector,~\eqref{eq:K} implies the desired description for the distribution of probabilities for the stationary measure. (In particular,~\eqref{eq:K} already describes a probability measure, with no need of normalization.)

Now, for the case of an even~$N$, the only part that changes is that the length $N$ cycle is now odd. To handle it, we take an $L$-th power root of minus unity, $\rt=\exp(\pi i /L)$, and instead of $C$ consider the operator $\rt C$, and instead of the base $e_k$ of $V$ we consider $\rt^k e_k$ and hence the base
$$
\tilde{v}_{k_1,\dots, k_N}:=\rt^{k_1+\dots+k_N} e_{k_1}\wedge \dots \wedge e_{k_N}. 
$$
Then again, the action of $\rt C\wedge E\wedge\dots \wedge E$ in this base becomes the action of the transition matrix~$T$; note that now for the jump from $k_N=L$ to~$1$ one gets two changes of sign, one from the length $N$ cycle, and another from $w^L=-1$:
$$
\rt^{k_1+\dots+k_{N-1}+L} e_{k_1}\wedge \dots \wedge (\rt C\cdot e_{L}) = \rt^{1+k_1+\dots+k_{N-1}} e_1\wedge e_{k_1}\wedge \dots \wedge e_{k_{N-1}}. 
$$

Now, the eigenvalues of $\rt C$ are $\rt \lambda_j$, thus the spectral radius (and the maximal real positive eigenvalue) of $T$ is equal to 
$$
r=\rt\lambda_{-m}+\dots + \rt\lambda_{m-1},
$$
where $N=2m$. Rewriting it as a sum of a geometric series with the denominator $e^{2\pi i/L} = \rt^2$, one gets the desired expression for the entropy
$$
e^h = r = \rt \frac{e^{2\pi i \cdot m/L} - e^{-2\pi i \cdot m /L}}{\rt^2 - 1}= \frac{\sin \frac{\pi N}{L}}{\sin \frac{\pi}{L}}.
$$
The same application of the formula for the Parry measure concludes the proof.

\subsection{The relation to the dimer and beads models}

Let us now approach the same question from a different angle, obtaining the relation to the dimer and beads models. 

\subsubsection{Freezing the jumps and the beads process}\label{ss:freeze}

Again, let $L,N$ be fixed. The correspondence that was described in Section~\ref{s:view} 
(see Figure~\ref{p:models}) allows to transform evolution of maya diagrams to the dimer 
covers of the corresponding hexagonal graph. This also applies to maya evolution on the 
circle, that is transformed to the dimer covers on the graph on the cylinder. However, this map is non-surjective: 
it becomes bijective if for the maya evolution we authorize (initially forbidden) absence of jumps and simultaneous jumps. 

The vertical extension method that we have used in the end of Section~\ref{s:view} 
to handle the simultaneous jumps would not work anymore in the circle case, 
as the total number of jumps in not anymore fixed. So instead we will use ``freezing''
techniques, imposing a ``tax'' on jumping. That is, we again consider a dimer 
configuration with a high number of levels (of some height~$M$), but this time, 
associate a (small) weight~$\eps$ to the ``jump'' edges, leaving all the others with the weight~$1$. 
Then, we take the weight of a dimer configuration to be the product of weights of dimers 
used, and choose a dimer cover with the probability proportional to its weight.

Consider first the limit where $M$ is chosen to grow as $M\sim \frac{\tau}{\eps}$, where $\tau$ is a constant. 
In this limit, we have the following
\begin{Lem}\label{l:no-simultaneous}
Whichever are the boundary (initial and final) conditions, the probability (that is, the proportion of total 
weight of configurations) of two jumps on the same level converges to zero as $\eps\to 0$. 
\end{Lem}
Before proving it formally, note that for any $n$ all the configurations with $n$ jumps, all at different levels, 
have the same probability (as they have the same weight~$\eps^n$). In particular, conditioning to a given $n$ 
gives the choice of moments of jumps that are uniformly chosen among ${M \choose n}$. In particular, 
rescaling the time $\eps$ times by denoting $t:=\eps k \in [0,\eps M]$ (where $k$ is the vertical coordinate), 
we see that this conditioning leads in the limit $\eps\to 0$ to the uniform choice of $n$ points on $[0,\tau]$.

We can then consider the bulk limit: make $\tau$ go to infinity and shift the origin to $\frac{\tau}{2}$ in the rescaled coordinates.
The jump places and (renormalized) moments then provide a cylinder analogue of the bead process, a random subset of $\bbZ_L \times \bbR$.
\begin{theorem}\label{t:cylinder}
This limit process is given by coupling a maximal entropy measure for the two-sided topological Markov chain and of a Poisson process on $\bbR$ of constant intensity, providing the jump moments. The intensity of the Poisson process is equal to $e^h$, where $h$ is the entropy of the Markov chain (given by Theorem~\ref{th:TASEP}).
\end{theorem}

\begin{proof}[Proof of Lemma~\ref{l:no-simultaneous}]
Consider the corresponding partition function $Z$, that is the sum of weights of all the configurations, and its part $Z_0$ that is given by the sum of weights of configurations with no simultaneous jumps. Let $W_n$ be the number of paths from the initial to the final configuration, consisting of~$n$ jumps. 
It suffices to show that as $\eps\to 0$, both $Z$ and $Z_0$ converge to the same (positive and finite) limit.

On one hand, we have
\begin{equation}\label{eq:eps-lower}
Z_0=\sum_n W_n \eps^n {M\choose n}.
\end{equation}
On the other hand, when we authorize configurations with simultaneous jumps, we can still enumerate them by a non-decreasing sequence of moments $1\le k_1\le\dots\le k_n\le M$, and the number of such sequences equals ${M+n -1 \choose n}$. Thus
\begin{equation}\label{eq:eps-upper}
Z\le \sum_n W_n \eps^n {M+n \choose n}.
\end{equation}
Note that for any fixed $n$
\[
W_n \eps^n {M \choose n} \sim W_n \eps^n \frac{M^n}{n!} = W_n \frac{(\eps M)^n}{n!}  \xrightarrow[\eps\to 0]{} W_n \frac{\tau^n}{n!},
\]
and the same applies for the terms of the second series. Hence, both series coverge termwise as $\eps\to 0$ to the series 
\[
\sum_n W_n \frac{\tau^n}{n!},
\]
that is convergent (and whose sum is strictly positive). To conclude the proof, it suffices thus to check that their convergence is uniform in $\eps$ in some neighbourhood of zero, $(0,\eps_0)$. To do so, we will provide an upper estimate for the terms of these series by a convergent series that does not depend on~$\eps$. 

Indeed, fix $R$ that is larger than the norm of the transition matrix of our Markov chain, then $W_n<R^n$ for all $n$. Now, for any $\eps>0$ if $n\le M$, we have 
$$
{M+n \choose n} \le {2M \choose n}  < \frac{(2M)^n}{n!},
$$
and the corresponding term does not exceed (once $M<\frac{2\tau}{\eps}$)
\[
W_n \eps^n {M+n \choose n} \le R^n \eps^n \frac{(2M)^n}{n!} < \frac{(4\tau R)^n}{n!};
\]
the term in the right hand side provides a convergent series that does not depend on $\eps$. On the other hand, if $M<n$, we have ${M+n \choose n}< {2n \choose n} < 2^{2n}$, and thus 
\[
W_n \eps^n {M+n \choose n} \le R^n \eps^n 2^{2n} = (4R\eps)^n < \frac{1}{2^n}
\]
once $\eps<\frac{1}{8R}$. Hence, both series $Z$ and $Z_0$ for all sufficiently small $\eps>0$ are bounded termwise by the series
\[
\sum_n \max(\frac{(4\tau R)^n}{n!}, \frac{1}{2^n}),
\]
that is convergent and does not depend on~$\eps$. Hence, their convergence is uniform as $\eps\to 0$, and this concludes the proof of the lemma.
\end{proof}

\begin{proof}[Proof of Theorem~\ref{t:cylinder}]
Note first that due to Lemma~\ref{l:no-simultaneous} the process that we obtain on $[-\frac{\tau}{2},\frac{\tau}{2}]\times \bbZ_L$ can be equivalently obtained by passing to the limit only from the configurations with no simultaneous jumps. 

Also from Lemma~\ref{l:no-simultaneous} and from its proof, for any given $\tau>0$ this limiting process can be described in the following way. 
First, one randomly chooses a number $\xi$ of jumps, in such a way that the probability of $\xi=n$ is proportional to $W_n \frac{\tau^n}{n!}$. Then, one of $W_\xi$ length $\xi$ paths satisfying the boundary conditions is chosen equiprobably, as well as a set of $\xi$ independently chosen points on $[-\frac{\tau}{2},\frac{\tau}{2}]$, giving the moments, at which (after putting them in the increasing order) the jumps following the chosen path will occur. 

Next, let us describe the ``average density'' of the jumps: we have the following lemma.
\begin{Lem}\label{l:proba}
As $\tau\to\infty$, the fraction $\frac{\xi}{\tau}$ between the (random) number of jumps $\xi$ and the total time $\tau$ converges in probability to the constant value~$e^h$.
\end{Lem}
\begin{proof}
Let $\rho$ stay for the spectral radius of the transition matrix of our topological Markov chain; then, $\rho=e^h$. If we had $W_n=\rho^n$, then the distribution of $\xi$ would follow exactly the Poisson law with the parameter~$\rho \tau$, and the statement of the lemma would be a mere Law of Large Numbers.

Now, our Markov chain is transitive. If it was also aperiodic, we would have $W_n\sim c \rho^n$ for some constant~$c$. However, it is not; it is easy to check that its minimal period is equal to $L$, the length of the circle. Thus, for any chosen boundary conditions there exists a residue $n_0$ such that the number $W_n$ of Markov chain paths of length $n$ behaves as 
\[
W_n \sim c \rho^n \quad \text{if } \, n \equiv n_0  \mod L,
\]
where $c$ is a constant (depending on the particular choice of the boundary conditions) and $W_n=0$ otherwise. The conclusion of the lemma then can be deduced from the ``pure exponent'' case. Indeed, the distribution of $\xi$ for a given $\tau$ can be obtained by a series of two operations. First, a Poisson random variable $\pi(\rho \tau)$ is conditioned to be congruent to $n_0 \mod L$. Then, for the obtained probability distribution the probability of each $n$ is multiplied by a bounded factor (corresponding to passing from $\rho^n$ to $W_n$). 

And both these operations do not affect the Law of Large Numbers conclusion. Indeed, the first one selects a part of lower-bounded probability (asymptotically $1/L$-th one, as $\tau\to\infty$), while the second one can change the quotient of probabilities of the events only by a bounded factor (and hence  also cannot break the ``with probability convergent to~1'' statement). Thus, we have the desired Law of Large Numbers: the quotient $\frac{\xi}{\tau}$ converges to $\rho$ in probability as $\tau\to\infty$.
\end{proof}

Now, selecting $n\sim \rho \tau$ uniformly distributed independent points on the interval $[-\frac{\tau}{2},\frac{\tau}{2}]$ converges as $\tau\to\infty$ to the Poisson process on the real line with the intensity~$\rho$. Thus the same holds if we average on a set of values of $n$ that $\xi$ takes with the probability convergent to $1$, on which $\frac{\xi}{\tau}\to\rho$.

Now, for any fixed interval $[a,b]$ on the real line consider the number $\xi_1$ of the jumps on $[-\frac{\tau}{2},a]$. Note that in probability $\xi_1$ tends to infinity, while its residue modulo $L$ is asymptotically uniformly distributed. 

For an aperiodic transitive topological Markov chain, the uniform distribution on paths with given boundary conditions in the bulk converges to the maximal entropy measure. Meanwhile, for a period~$L$ transitive Markov chain the accumulation points of such uniform distributions are the $L$ components of the maximal entropy measure that are permuted by the dynamics. However, as we take here the ``observation window'' $[a,b]$ that is separated from the fixed boundary $-\frac{\tau}{2}$ by the random number of steps $\xi_1$ that has all the residues $\mod L$ asymptotically equiprobable as $\tau\to\infty$, these permuted components are being averaged and one gets exactly the maximal entropy measure.
\end{proof}

\subsubsection{Bead process' kernel}\label{ss:beads}

We would not go into this alternate approach if it wouldn't lead to some interesting connections. Namely, let us study the random dimer covers that have already appeared in Sec.~\ref{ss:freeze} via the standard methods, that is, via the Kasteleyn theorem. 

Again, let $\eps,M$ be fixed, and we consider a chosen dimer partition of the corresponding graph of height $M$ with the weights $\eps$ on the ``jump'' edges that is chosen randomly in such a way that its probability is proportional to the weight of the configuration (in other words, with respect to the corresponding Gibbs measure).

Let us recall the statement of the Kasteleyn Theorem~\cite{Kast1, Kast2}. Let a planar bipartite graph with a weighted adjacency matrix $W_0=(\wt_{bw})$ be given. Fix additional factors $(\alpha_{bw})$, such that for any face of the graph, formed by vertices $b_1,w_1,\dots,b_k,w_k$, one has 
\begin{equation}\label{eq:frac}
\frac{\alpha_{b_1 w_1}\alpha_{b_2 w_2}\dots \alpha_{b_k w_k}}{\alpha_{b_1 w_2}\alpha_{b_2 w_3}\dots \alpha_{b_k w_1}}=(-1)^{k-1};
\end{equation}
at least one such choice always exists (it follows from the planarity of the graph). Then for all possible dimer covers $(j,{\sigma(j)})$ of the graph the products 
\[
\sign(\sigma) \cdot \prod_j \alpha_{j {\sigma(j)}}
\]
take the same value $\mathbf{a}$. This implies that the determinant of the matrix $W=(w_{bw}\alpha_{bw})$ equals to the product $\mathbf{a}\cdot Z$, where $Z$ is the corresponding statistical sum, as all the dimer covers contribute to the determinant with their weights times $\mathbf{a}$ (and the signs cancel out). Hence the probability of dimers $(b_{1},w_{1}),\dots, (b_{k},w_{k})$ being chosen for a Gibbs-random configuration is equal to
\begin{equation}\label{kasteleyn1}
P((b_i,w_i)_{i=1,\dots,k} \text{ chosen}) = \prod_{i=1}^k (\alpha_{b_i w_i}\wt_{b_i w_i}) \cdot \det (K_{w_jb_i})_{i,j=1,\dots,k},
\end{equation}
where $K=W^{-1}$ is the inverse matrix.

We are going to apply this theorem to our graph, that is bipartite and planar. Indeed, it naturally embeds into a cylinder, which can be sent to the plane using the polar coordinates. Under this embedding, \emph{almost} all the faces of the graph become hexagons. However, there are two exceptions: the inner and the outer faces, that have $2L+2(L-N)=4L-2N$ sides each. The choice of the factors $\alpha_{bw}$ will thus depend on the parity of $N$.

Namely, for odd $N$ we can take all the $\alpha$'s to be equal to~$1$: all the faces have number of faces of the form $2(2k+1)$. However, it turns out that the following choice will simplify the later computations: we take
\begin{equation}
\alpha_{bw}=\begin{cases}
1, & \text{if } bw \text{ is a jump edge}\\
1, & \text{if it is a ``stone stays'' edge}\\
-1, & \text{if it is a ``no stone'' edge}.
\end{cases}
\end{equation}
It is easy to check that this choice satisfies the condition~\eqref{eq:frac}: the fractions in its left hand side have the same number of $(-1)$'s in the numerator and denominator. In the same way, for even $N$ we handle the inner and outer faces in the most ``rotationally symmetric'' way, taking
\begin{equation}
\alpha_{bw}=\begin{cases}
\rt=\exp(\pi i /L), & \text{if } bw \text{ is a jump edge}\\
1, & \text{if it is a ``stone stays'' edge}\\
-1, & \text{if it is a ``no stone'' edge}.
\end{cases}
\end{equation}
Indeed, for such a choice one gets in the right hand side of~\eqref{eq:frac} the fraction $\frac{-\rt}{-\rt}=1$ for any hexagonal face, and $\rt^L=-1$ for the inner and outer ones, thus satisfying the assumptions of the Kasteleyn theorem.

Now, in our weighted adjacency matrix $W_0$ there are edges of two different weights: $1$ and~$\eps$. This (after the application of the Kasteleyn theorem) leads us to the consideration of two \emph{different} possible determinantal-type processes. Namely, we can consider:
\begin{itemize}
\item \textbf{The presence of stones at given times and positions}; in the limit $\eps\to 0$, their presence is given by the corresponding ``stone stays in the place'' edges (the probability of a jump at any particular time tends to zero). The product of weights of these edges is equal to~$1$, and so the corresponding probability tends to the corresponding minor of the limit of the matrix $K=W^{-1}$.
\item \textbf{The positions and times of the jumps}, in other words, the corresponding bead process. As the jump edges have weight $\eps$, for a $k$-edges configuration its probability is given by a product 
\[
\eps^k \cdot \det (K_{w_j b_i})_{i,j=1,\dots,k}
\]
for odd $N$ (and with an additional $\rt^k$ in front for an even~$N$).
The factor $\eps^k$ corresponds to the density interpretation (we rescale the time by $\eps$), and in the limit $\eps\to 0$ we get a continuous-time determinantal process: the \emph{densities} are determinants of the corresponding minors of the matrix~$K$.
\end{itemize}

Considering the limit in the second sense, we will see that this jump edges process converges to a circle-based analogue of the beads process studied in~\cite{Sun, Boutillier}. Passing then to the limit $L\to\infty$ allows to recover exactly their beads' process, providing an alternate viewpoint on its correlation kernel (see~\cite[Eq.~(9)]{Boutillier}).

\begin{figure}[!h!]
\begin{center}
\includegraphics[scale=0.9]{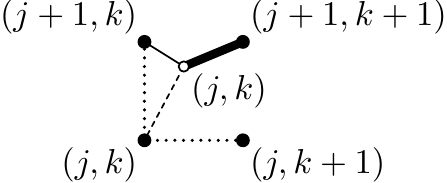}
\end{center}
\caption{``No-stone'' edge (dashed line), ``stone staying'' edge (simple line) and ``jumping'' edge (bold line), as well as the indices of the corresponding black and white vertices.}\label{f:indices}
\end{figure}

To do all of this rigorously, let us first consider the behaviour of such a configuration for a fixed~$\eps$. Let $\Mm,\Mp$ be given, and  $\Mm< j\le \Mp$ and $1\le k\le L$ be the time- and circle-wise coordinates respectively. We will use the conventions from Section~\ref{s:view}: a white vertex with the coordinates $(k,j)$ is joined with the black vertices with the coordinates $(k,j)$, $(k,j+1)$ and $(k+1,j+1)$ (see Fig.~\ref{f:indices}). Let us group the vertices into the (size $L$) blocks with the same  $j$ (time) coordinate. The matrix $W$ then takes the form
\begin{equation}
W(\eps,\Mm,\Mp)=\left( \begin{matrix}
-U_1 & B & 0 & \dots & 0 & 0\\
0 & -E & B & \dots & 0 & 0\\
0 & 0 & -E & \dots & 0 & 0\\
\hdotsfor{6} \\
0 & 0 & 0 & \dots & -E & B\\
0 & 0 & 0 & \dots & 0 & -U_2^T,
\end{matrix}\right)
\end{equation}
where $B=E+\eps C$ if $N$ is odd, and $B=E+\rt \eps C$ if $N$ is even, and the matrices $U_1$ and $U_2$ of size $L\times (L - N)$ correspond to the initial and final boundary conditions (consisting of ones and zeros only). 

Now, let us calculate the inverse matrix $W(\eps,\Mm,\Mp)^{-1}$: fix some $(k,j)$ and consider the vector $u=u^{(j,k)}$ that is send by $W(\eps, \Mm,\Mp)$ to the base vector with the only $1$ at the moment $j$ of time at the place~$k$. The above block decomposition allows then to write this equation as 
\begin{equation}\label{eq:A-u}
\left( \begin{matrix}
-U_1 & B & 0 & \dots & 0 & 0\\
0 & -E & B & \dots & 0 & 0\\
0 & 0 & -E & \dots & 0 & 0\\
\hdotsfor{6} \\
0 & 0 & 0 & \dots & -E & B\\
0 & 0 & 0 & \dots & 0 & -U_2^T
\end{matrix}\right)
\left( \begin{matrix}
[u_{\Mm}]\\
u_{\Mm+1}\\
u_{\Mm+2}\\
\vdots \\
u_{\Mp-1}\\
u_{\Mp}
\end{matrix}\right) = \left( \begin{matrix}
0\\
\dots\\
0\\
e_k \\
0\\
\dots\\
0
\end{matrix}\right);
\end{equation}
here $[u_{M_-}],u_{M_{-}+1},\dots,u_{M_+}$ are the blocks of $u$, and in the right hand side the base vector $e_k$ is placed at $j$-th size $L$ block. We denote the first component $[u_{\Mm}]$ (and not by~$u_{\Mm}$), because it is of size~$L-N$ instead of~$L$, and define $u_{\Mm}:=U_1([u_{\Mm}])\in \bbR^L$. 

The block lines other than the the last one of the system~\eqref{eq:A-u} become a recurrent relation
\begin{equation}\label{eq:relation}
\begin{cases}
-u_{i} + B u_{i+1} = 0, & i\neq j,  \quad \Mm< i<\Mp. \\
-u_{j} + B u_{j+1} = e_k.
\end{cases}
\end{equation}
The first and the last lines become the ``boundary conditions'' $u_{\Mm}\in V_-$, $u_{\Mp}\in V_+$, where $V_-:=U_1 (\bbR^{L-N})$ and $V_+ := \ker U_2^T$
are $L-N$ and $N$-dimensional subspaces respectively. The relation~\eqref{eq:relation} implies that 
$$
u_{j}=B^{-(j-\Mm)} u_{\Mm}, \quad B u_{j+1} = B^{\Mp-j} u_{\Mp}.
$$

Hence we are decomposing the vector $e_k$ as a sum $e_k= - u_- + u_+$, where 
\begin{equation}\label{eq:decomposition}
u_-:=u_j\in V_{-,j}:= B^{-(j-\Mm)} V_-, \quad  u_+:=Bu_{j+1} \in V_{+,j}:=B^{\Mp-j} V_+.
\end{equation}

Now, $(i,k')$-th element of $W^{-1}$ is the $k'$-th coordinate of the vector $u_i$, that is equal to
\begin{equation}\label{eq:u-W}
u_i=\begin{cases}
B^{j-i} u_-,  & i\le j \\ 
B^{-(i-j)} u_+, & i>j.
\end{cases}
\end{equation}

Note that the matrix $W(\eps, \Mm,\Mp)$ might be degenerate for small $M:=\Mp-\Mm$, when there are no possible length $M$ paths joining the given initial and boundary conditions. Actually, the above arguments show that matrix $W$ is invertible if and only if the subspaces~$B^M V_-$ and $V_+$  are transversal (and the corresponding $L\times L$ determinant is easily seen to be equal to $\det W$).

\begin{proposition}\label{p:M-limit}
For any $\eps$ sufficiently small, as $\Mm\to-\infty$, $\Mp\to +\infty$, the elements of the matrix $W(\eps,\Mm,\Mp)^{-1}$ pointwise converge to those given by $K(j,k;j',k')= K(j'-j,k,k')$, where, considering $K(i,\cdot,\cdot)$ as a $L\times L$ matrix, one has
$$
K(j,\cdot,\cdot) = \begin{cases}
- B^{-j} P_-, & j\le 0, \\ 
B^{-j} P_+, & j> 0.
\end{cases}
$$
Here $P_+$ is the projector on the space $\hV_+$ spanned by $N$ consecutive Fourier harmonics, from $-m$-th to $m$-th for odd $N=2m+1$ and from $-m$-th to $m-1$-th for even $N=2m$, and $P_-=E-P_+$ is the projector on its orthogonal complement $\hV_-$, spanned by the $L-N$  complementary ones.
\end{proposition}

\begin{proof}
Due to the relations~\eqref{eq:decomposition} and \eqref{eq:u-W} it suffices to show that the spaces $V_{-,j}$ and $V_{+,j}$ converge in the setting of the proposition respectively to $\hV_-$ and~$\hV_+$. Such a convergence is quite natural to expect, as $\hV_+$ is the span of $N$ eigenvectors of $B$ with the largest in absolute value eigenvalues, while $\hV_-$ is the span of $L-N$ smallest ones. 

\begin{figure}[!h!]
\begin{center}
\includegraphics[scale=0.8]{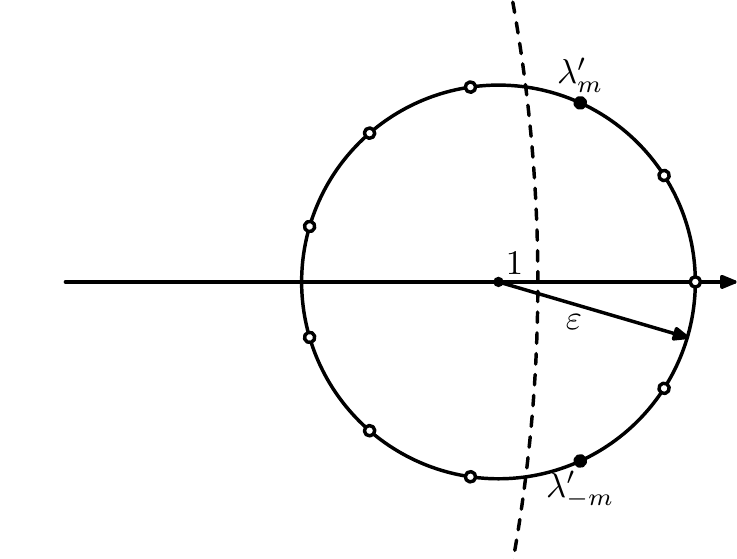} \quad 
\includegraphics[scale=0.8]{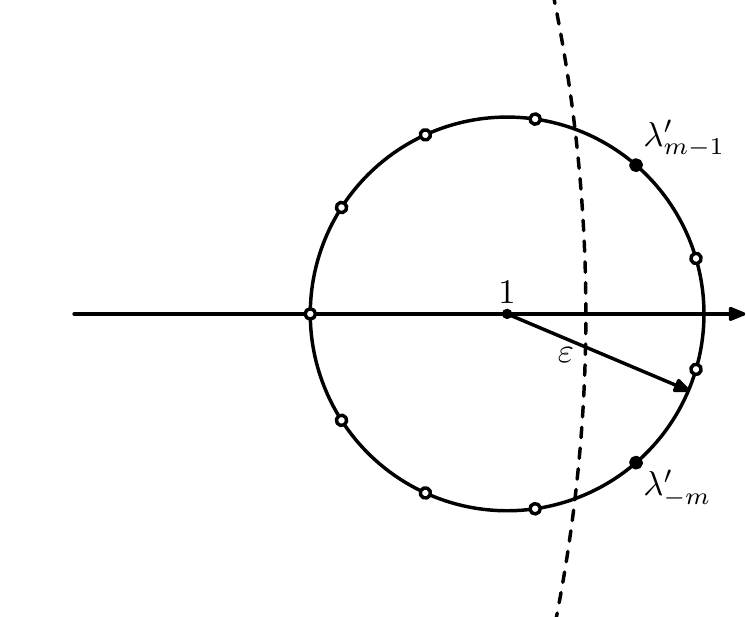}
\end{center}
\caption{Eigenvalues $\lambda_k'$, separated into groups of $N$ and $L-N$ by their absolute value for the cases of odd (left) and even (right)~$N$}\label{f:evals}
\end{figure}

To show such a convergence formally, we start with the study of $V_{j,+}$, and consider first the case of $N$ odd, $N=2m+1$. For the action of $B=E+\eps C$ on $N$-dimensional subspaces of $V=\bbR^L$, let us pass to the Plucker coordinates, considering the action of $\wedge^N B$ on the space $\wedge^N V$. Take the base of $\wedge^N V$ formed by 
\[
v_{k_1,\dots,k_N}=e_{k_1}\wedge \dots\wedge e_{k_N},  \quad k_1<\dots<k_N.
\]
Then, in the same way as in Section~\ref{ss:proof}, the action of $\wedge^N B$ in this base is given by a matrix with \emph{non-negative} elements, and there exists a sufficiently large power of $\wedge^N B$ that has all its elements strictly positive. This implies that on the projective space, the $\wedge^N B$-iterations of all the base vectors $v_{k_1,\dots,k_N}$ converge to the direction of the highest absolute value eigenvector of this operator. 

At the same time, as $B$ commutes with the rotation $C$, its eigenvectors (in $\bbC^L$) are the Fourier harmonics $\sum_k \exp(2\pi i k r/L) e_k, \quad r\in \bbZ_L$, with the corresponding eigenvalues 
\[
\lambda'_r=1+\eps \zeta_r, 
\]
where for odd $N$ we denote $\zeta_r:=\exp(2\pi i r/L)$ the eigenvalues of the rotation~$C$.

The $N$ largest in absolute values are the ones corresponding to $r=-m,\dots,m$, that are the base of~$\hV_+$, and we have thus obtained the desired convergence of $V_{+,j}$ to $\hV_+$. 

Now, if $N$ is even, $N=2m$, again as in Section~\ref{ss:proof} we consider the base 
\[
e'_1 = e_1, \quad e'_2 = \rt e_2, \quad e'_3 = \rt^2 e_3, \quad \dots, \quad e'_L = \rt^{L - 1} e_{L}. 
\]
Then one has 
\[
\rt C e'_i = \begin{cases}
e'_{i+1}, & i<L
\\
-e'_{1}, & i=L
\end{cases}
\]
Hence, for $B=E+\eps \rt C$ the operator $\wedge^N B$ again acts on the corresponding base 
\[
v'_{k_1,\dots,k_N}= e'_{k_1}\wedge\dots\wedge e'_{k_N}, \quad k_1<\dots<k_N
\]
as a matrix with non-negative elements (the signs cancel out if $e'_1$ occurs out of $e'_L$), and has a power whose elements are strictly positive.

We thus again get the convergence of directions of $\wedge^N B$-iterations of any of the base vectors under the to the direction of the highest weight eigenvector. The eigenvectors of $B$ are again the Fourier harmonics, with the eigenvalues 
\[
\lambda'_{r} = 1+ \eps \zeta_r,
\]
where for the even $N$ we denote by $\zeta_r:=\exp(2\pi i (r+1/2)/L)$ the eigenvalues of $\rt C$. The $N=2m$ largest in absolute value are $\lambda'_{-m},\dots,\lambda'_{m-1}$, and the corresponding eigenvectors (Fourier harmonics) span the space $\hV_+$. We have obtained the desired convergence of $V_{+,j}$ to~$\hV_+$.

Now, in both these cases ($N$ odd or even) the leading eigenvector $\beta$ of $\wedge^N B$ (that is the Plucker coordinates of $\hV_+$) is a vector with all strictly positive coordinates. This implies that the space $\hV_+$ is transversal to any of the $N-L$-dimensional coordinate subspaces (spanned by $L-N$ base vectors). Indeed, for any such subspace the wedge product $\beta\wedge e_{k_1}\wedge \dots\wedge e_{k_{L-N}}$ is equal to $\beta_{k'_1,\dots,k'_N} \, e_1\wedge \dots \wedge e_L$, where $k'_1,\dots,k'_N$ are the complementary coordinates to $k_1,\dots, k_{L-N}$, and (as the Plucker coordinate $\beta_{k'_1,\dots,k'_N}$ is strictly positive) thus is nonzero. This transversality implies that for any such coordinate subspace, in particular, for the space $V_-$, its $B^{-1}$-iterations will converge to the space $\hV_-$ spanned by the $L-N$ eigenvectors of $B$ with the least norm of the eigenvalues. 

\end{proof}

\begin{rem}
As the matrix $B$ commutes with the circle rotation $C$, and as the Fourier transform diagonalizes it with the eigenvalues $\lambda'_r$ for the Fourier harmonic $v_r=(e^{-2\pi i k r/L})_{k\in\bbZ_L}$, we can consider the operator $K(j;\cdot,\cdot)$ as a composition of four operators: 
\begin{itemize}
\item Fourier transform $F$;
\item Projection that leaves only one of two complementary groups of adjacent Fourier coefficients, of length $N$ (that is, $-m,\dots, m$ or $-m,\dots, m-1$ depending on if $N$ is odd or even) for positive~$j$ and of length $L-N$ (that is, $m+1,\dots, L-(m+1)$ or $m,\dots, L-(m+1)$ depending on if $N$ is odd or even) for negative~$j$; 
\item Diagonal operator of multiplication by $(\lambda'_r)^{-j}$
\item Inverse Fourier transform~$F^{-1}$.
\end{itemize}
\end{rem}

\begin{cor}
Again, as the matrix $B$ commutes with the circle rotation $C$, we actually have $K(j;k,k')=K(j,k-k')$, where 
\begin{equation}\label{eq:K-finite}
K(j,k) = 
\begin{cases}
 \frac{1}{L}\sum_{r=-m}^m (\lambda'_r)^{-j} e^{-2\pi i k r/L}, & j> 0 \\ 
- \frac{1}{L}\sum_{r=m+1}^{L-m-1} (\lambda'_r)^{-j} e^{-2\pi i k r/L}, & j\le0 \\ 
\end{cases}
\end{equation}
for odd $N=2m+1$ and 
\begin{equation}\label{eq:K-finite-even}
K(j,k) = 
\begin{cases}
 \frac{1}{L}\sum_{r=-m}^{m-1} (\lambda'_r)^{-j} e^{-2\pi i k r/L}, & j> 0 \\ 
- \frac{1}{L}\sum_{r=m}^{L-m-1} (\lambda'_r)^{-j} e^{-2\pi i k r/L}, & j\le0 \\ 
\end{cases}
\end{equation}
for even $N=2m$.
\end{cor}

Now, let us pass to the limit as $\eps\to 0$, with the simultaneous time-rescaling by considering $t=\eps j$. Note that even if this order of limits is slightly different from the one in 
Sec.~\ref{ss:freeze} (where we passed to the limit first as $\eps\to 0$ on the time intervals 
$\sim [-\frac{\tau}{2\eps}, \frac{\tau}{2\eps}]$ and then to the limit as $\tau\to\infty$), 
we still get the same random process as a limit: 
\begin{Lem}
Limit of the processes in Proposition~\ref{p:M-limit} as $\eps\to 0$ coincides with the one described in Theorem~\ref{t:cylinder}.
\end{Lem}
\begin{proof}
Let $\tau$ be fixed. Then, once $\Mm<-\frac{\tau}{2\eps}$ and $\Mp>\frac{\tau}{2\eps}$, due to the Gibbs property we can consider the random configuration inside $[-\frac{\tau}{2\eps},\frac{\tau}{2\eps}]\times \bbZ_L$ as being sampled in two steps: first the boundary conditions on the levels $\pm \frac{\tau}{2\eps}$, and then the inside part as a Gibbs measure conditional to these boundary conditions. Thus, the restriction of the Gibbs measure on the domain $[-\frac{\tau}{2\eps},\frac{\tau}{2\eps}]$ can be seen as a mix of the measures discussed in Sec.~\ref{ss:freeze} (as the boundary conditions are varied).

Now, as $\eps\to 0$, $\eps$-rescaled images of all these measures converge to the same process described in Theorem~\ref{t:cylinder}, and hence the same applies to their average (whichever were the averaging coefficients).
\end{proof}

We can now pass to the limit either in the probabilities of the stones being present, or for the position and moments of their jumps. For the stones, as the probability of their presence is given by an exact determinantal formula for any fixed $\eps>0$, we have the same kind of formula for their limit:
\begin{theorem}\label{stones}
For the limit process in Theorem~\ref{t:cylinder}, the probability that the \emph{stones} are present at positions $k_1,\dots,k_n$ at times $t_1,\dots,t_m$ is equal to the determinant 
\[
\det (\tilde{K}(t_a-t_b, k_a - k_b)_{a,b=1,\dots,n}),
\]
where 
\begin{equation}\label{eq:lim-K}
\tilde{K}(t,k)=
\begin{cases}
 \frac{1}{L}\sum_{r=-m}^m e^{-t \zeta_r} e^{-2\pi i k r/L}, & t> 0 \\ 
 -\frac{1}{L}\sum_{r=m+1}^{L-m-1} e^{-t \zeta_r} e^{-2\pi i k r/L}, & t\le0 \\ 
\end{cases}
\end{equation}
for odd $N=2m+1$ and 
\begin{equation}\label{eq:lim-K-even}
\tilde{K}(t,k)=
\begin{cases}
 \frac{1}{L}\sum_{r=-m}^{m-1} e^{-t \zeta_r} e^{-2\pi i k r/L}, & t> 0 \\ 
- \frac{1}{L}\sum_{r=m}^{L-m-1} e^{-t \zeta_r} e^{-2\pi i k r/L}, & t\le0 \\ 
\end{cases}
\end{equation}
for even $N=2m$.
\end{theorem}

\begin{cor} \label{cor:measure}
Take all the $t_i$ equal. Then, what we get is a distribution of probabilities for the configurations of stones at a single moment of time, and Theorem~\ref{stones} states that this is a determinantal point process with the kernel given by the projection operator on $N$ adjacent Fourier harmonics. This re-proves the statement of Theorem~\ref{th:TASEP} from the determinantal processes point of view.
\end{cor}

In the same way, consideration of the positions and moments of the jumps gives
\begin{theorem}\label{beads}
For the limit process in Theorem~\ref{t:cylinder}, the common density of the probability for the jumps at $(k_1,t_1),\dots,(k_n,t_n)$ is equal to the determinant 
\begin{equation}\label{eq:tK-jumps}
\det (\tilde{K}(t_a-t_b, k_a - k_b - 1)_{a,b=1,\dots,n})
\end{equation}
for odd $N$ and to the determinant 
\begin{equation}\label{eq:tK-jumps1}
\det (\rt \tilde{K}(t_a-t_b, k_a - k_b - 1)_{a,b=1,\dots,n})
\end{equation}
for even $N$.

\end{theorem}

Note (see Figure~\ref{f:indices}) that the jump edges join a white vertex with the coordinates $(j,k)$ to the black one with the coordinates $(j+1,k+1)$, and this space-shift by $1$ leads to the $-1$ added to the difference of $k$ in~\eqref{eq:tK-jumps} and \eqref{eq:tK-jumps1}.

Next one can remark that the function $\tilde{K}$ given by~\eqref{eq:lim-K} is not perfectly suitable for the determinantal processes study: its asymptotics allows an exponential growth to the past or to the future. However, there is again a freedom in the choice of the gauge (similar to the one that we have already used for the jump edges): we can conjugate the matrix $K$ that we obtain for a finite $\eps$ by the diagonal matrix with the elements~$(c')^j$, where $c'$ is chosen so that 
\begin{equation}\label{c'-cond}
|\lambda'_{m+1}|<c'<|\lambda'_{m}|. 
\end{equation}
This replaces the kernel~\eqref{eq:K-finite} with
\begin{equation}\label{eq:K-finite-c}
K_{c'}(j,k) = 
\begin{cases}
 \frac{1}{L}\sum_{r=-m}^m (\lambda'_r/c')^{-j} e^{-2\pi i k r/L}, & j> 0 \\ 
 -\frac{1}{L}\sum_{r=m+1}^{L-m-1} (\lambda'_r/c')^{-j} e^{-2\pi i k r/L}, & j\le0, \\ 
\end{cases}
\end{equation}
that is now exponentially decreasing in both $j\to +\infty$ and in $j\to-\infty$. 

Now, as we pass to the limit as $\eps\to 0$, it is natural to take $c'=1+\eps c$ (so that its $j=\frac{t}{\eps}$-th power tends to the exponent). The condition~\eqref{c'-cond} then becomes 
\begin{equation}\label{c-gamma}
\Re \zeta_{m+1}<c<\Re\zeta_m,
\end{equation}
and such a choice of $c$ after passing to the limit leads to the kernel
\begin{equation}\label{eq:lim-K-c}
\tilde{K}_c(t,k)=
\begin{cases}
 \frac{1}{L}\sum\limits_{r:\, \Re \zeta_r >c} e^{-t (\zeta_r-c)} e^{-2\pi i k r/L}, & t> 0 \\ 
 -\frac{1}{L}\sum\limits_{r: \,\Re \zeta_r <c} e^{-t (\zeta_r-c)} e^{-2\pi i k r/L}, & t\le0 \\ 
\end{cases}
\end{equation}
for the ``finite-circle bead process''  that exponenitally decreases in both past and future.

A final remark is that passing to the limit as $L\to \infty$ with $N/L\to \rho$ transforms the kernel~\eqref{eq:lim-K-c} to the one appearing in~\cite[Eq.~(9)]{Boutillier} under time renormalization and change of parametrization. Indeed, as $L\to\infty$, the eigenvalues $\zeta_m$, $\zeta_{m+1}$ tend to the common limit $g_{\infty}:=e^{\pi i \rho}$, and hence the limit value of $c$'s (from passing to the limit in~\eqref{c-gamma}) is 
\[
c_{\infty}:=\cos \pi \rho.
\] 

The sums in the kernel~\eqref{eq:lim-K-c} tend to the integral over the corresponding arcs of the unit circle; the limit kernel thus is
\begin{equation}
\tilde{J}_{beads}(t,k)=
\begin{cases}
\frac{1}{2\pi} \int_{-\pi\rho}^{\pi\rho} e^{-t (\zeta-c_{\infty})} e^{-i\varphi(k-1)} \,d\varphi, & t> 0, \\ 
-\frac{1}{2\pi} \int_{\pi\rho}^{2\pi - \pi\rho} e^{-t (\zeta-c_{\infty})} e^{-i\varphi(k-1)} \,d\varphi, & t\le0, \\ 
\end{cases}
\end{equation}
where $\zeta = e^{i\varphi}$. Changing the integration variable to $\zeta$, with $d\varphi=\frac{d\zeta}{i\zeta}$, we get:

\begin{equation}\label{Jt-beads}
\tilde{J}_{beads}(t,k)=
\begin{cases}
\frac{1}{2\pi i} \int_{I_1} e^{-t (\zeta-c_{\infty}))} \zeta^{-k} \,d\zeta, & t> 0 \\ 
-\frac{1}{2\pi i} \int_{I_2} e^{-t (\zeta-c_{\infty}))} \zeta^{-k} \,d\zeta, & t\le0, \\ 
\end{cases}
\end{equation}
where $I_1 = \exp ( i [-\pi \rho,\pi \rho])$ and $I_2= \exp ( i [\pi \rho,2\pi -\pi \rho])$ are two complementary arcs of the unit circle joining $\overline{g}_{\infty}$ and $g_{\infty}$ (see Fig.~\ref{f:zeta}).

\begin{figure}[!h!]
\begin{center}
\includegraphics[scale=0.8]{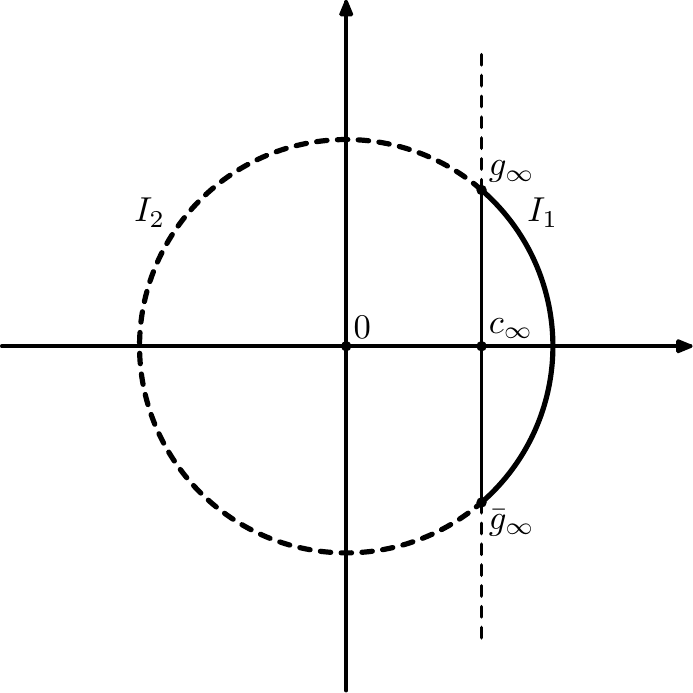}
\end{center}
\caption{Integration paths}\label{f:zeta}
\end{figure}

Now, let $\rho<1/2$, and hence $c_{\infty}>0$. The function under the integral is holomorphic in $\bbC\setminus \{0\}$, and hence the integral over the arc $I_1$ can be replaced with the integral along a straight segment; denoting $\zeta=c_{\infty}+i \phi \sqrt{1-c_{\infty}^2}$ transforms this integral into 
$$
\sqrt{1-c_{\infty}^2} \cdot \int_{[-1,1]} e^{-it\phi \sqrt{1-c_{\infty}^2}} (c_{\infty}+i \phi \sqrt{1-c_{\infty}^2})^{-k} d\phi.
$$
In the same way, as the function for $t<0$ is exponentially decreasing in the left half-plane, the integral over the arc $I_2$ equals to the integral over $[g_{\infty},c_{\infty}+i\infty] \cup [c_{\infty}-i\infty, \overline{g}_{\infty}]$, and thus to 
$$
-\sqrt{1-c_{\infty}^2} \cdot \int_{\bbR \setminus [-1,1]} e^{-it\phi \sqrt{1-c_{\infty}^2}} (c_{\infty}+i \phi \sqrt{1-c_{\infty}^2})^{-k} d\phi.
$$
Taking $\gamma:=c_{\infty}$ and rescaling the time $\sqrt{1-c_{\infty}^2}$ times, we obtain the kernel, appearing in~\cite[Eq.~(9)]{Boutillier}.

\section{Young Through The Looking Glass}\label{s:Young}

The study of the Plancherel measures $\mu_n$ on the spaces $\bbY_n$ in the seminal paper~\cite{BOO}, 
was based on their \emph{poissonization}. Namely, for a fixed $\theta>0$, the authors consider the mixed sum $\sum_n \frac{e^{-\theta^2} (\theta^2)^n}{n!} \mu_n$ that is a measure on the space of all Young diagrams $\bbY=\bigsqcup_n \bbY_n$. Then, the authors show that these measures are determinantal ones, with kernels that are explicitly specified. 

It is interesting to note, that the perfect matchings encoding allows to explain, \emph{why} these measures are determinantal. The author thanks G.~Merzon and V.~Kleptsyn for these remarks.

Namely, consider the hexagonal graph corresponding to the encoding of a path in the Young graph, with some ``target diagram''~$\lambda$ (see Fig.~\ref{f:encodings-tilings}). Denote this graph $\Gamma_{\lambda,M}$, where~$M$ is the height of the graph. The target diagram is then specified by upper right ``green'' edges atop of the last row, being the maya encoding for~$\lambda$  (namely, these edges attachments correspond to the empty holes). 

Let us remove these edges, add a mirror image of the same graph, and join it with the initial one by 
vertical edges at \emph{all} the vertices: see Fig.~\ref{f:mirror}, right. Denote this graph by~$\widehat{\Gamma}_{M}$. Then, a perfect matching on the resulting graph is a pair of length~$M$ paths in the maya diagram encodings, heading towards the same ``target'' diagram $\lambda$, encoded by the matched pairs that cross the mirror, where on each step each stone either stays or jumps forward. An example of such matching is on Fig.~\ref{f:mirror}, right, with the encoded jumps shown on  Fig.~\ref{f:mirror}, left.

As earlier, let us equip the ``jump'' edges with a very small weight~$\eps$, while taking the height of this graph to be $2M\sim \frac{2}{\eps}\theta$. Then (in the same way as before), as $\eps\to 0$, for a fixed width and growing height graph, the total probability of a simultaneous jump (that is, of existence of a level at which two stones jump simultaneously) tends to 0. 

For any given $n$-cell diagram $\lambda$, the perfect matchings in the graph $\Gamma_{\lambda,M}$, that do not encode any simultaneous jumps, are in one-to-one correspondence with a pair of a path to~$\lambda$ in the Young graph (describing the order of the jumps) and of the set of rows when these jumps (in this order) occur. The weight of each such matching is $\eps^n$, there are $\dim \lambda$ different paths towards $\lambda$ in the Young graph, and hence (as $\eps\to 0$ and accordingly $M\to \infty$) their total weight asymptotically behaves as 
$$
{M \choose n} \eps^n \cdot \dim \lambda \sim \frac{(M\eps)^n}{n!}  \dim \lambda \to \frac{\theta^n}{n!} \dim \lambda.
$$

The perfect matching in $\widehat{\Gamma}_M$ is a pair of two such matchings with the same target diagram~$\lambda$, and hence the total weight of matchings corresponding to a given $\lambda$ asymptotically behaves as

\begin{figure}
\begin{center}
\includegraphics[scale=0.9]{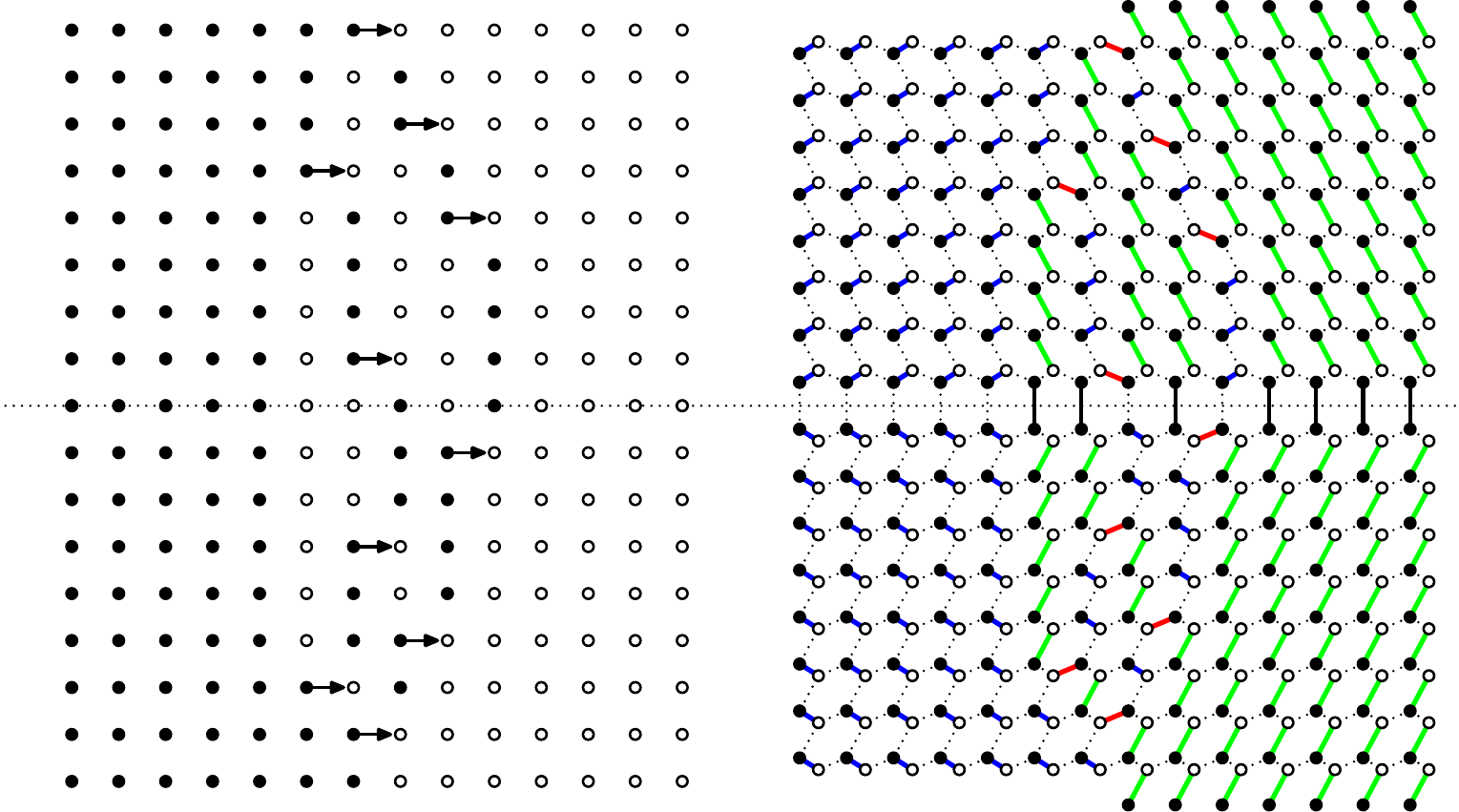}
\end{center}
\caption{Domino tiling for the Poissonization of the Plancherel measure}\label{f:mirror}
\end{figure}

\begin{equation}
\left(\frac{\theta^n}{n!} \dim \lambda\right)^2 =  \frac{\theta^{2n}}{n!} \cdot \frac{\dim^2 \lambda}{n!} = 
\frac{(\theta^2)^n}{n!} \cdot \mu_n(\{\lambda\}).
\end{equation}

Thus, normalizing the limiting distribution to the probability one, one will get  the poissonization of the Plancherel measures, restricted to the set of diagrams that fit to a given \emph{width}. Finally, as the width tends to the infinity, one gets exactly the poissonization of all the Plancherel measures.

On the other hand, the normalized probability distribution that comes from a perfect matching on a weighted planar bipartite graph is known to be determinantal (due to Kasteleyn-type arguments).
Moreover, as a side remark the same argument explains why the width-restricted (on one or on both sides) poissonizations are also determinantal.

\section{Acknowledgments} 

The author would like to thank Vadim Gorin, Alexey Bufetov, Leonid Petrov, Greta Panova, Alejandro Morales, Igor Pak, Christophe Dupont, Grigory Merzon and Victor Kleptsyn for their interest to the work and helpful discussions.

\end{document}